\documentclass[11pt,a4paper]{article}
\usepackage{mathrsfs}
\usepackage{amsmath,amsthm,color,eepic}
\usepackage{amssymb} \usepackage{verbatim}

\theoremstyle{theorem}
\newtheorem{theorem}{Theorem}[section]

\newtheorem{lemma}{Lemma}[section]

\theoremstyle{definition}

\newcounter{mathitem}

\begin{document}

\title{\bf Faithful subgraphs and Hamiltonian circles of infinite graphs
\thanks{Supported by NSFC (11601429). E-mail: libinlong@mail.nwpu.edu.cn.}}
\date{}

\author{Binlong Li\\[2mm]
\small Department of Applied Mathematics, Northwestern Polytechnical University,\\
\small Xi'an, Shaanxi 710072, P.R. China} \maketitle

\begin{center}
\begin{minipage}{140mm}
\small\noindent{\bf Abstract:} A circle of an infinite locally
finite graph $G$ is the imagine of a homeomorphic mapping of the
unit circle $S^1$ in $|G|$, the Freudenthal compactification of $G$.
A circle of $G$ is Hamiltonian if it meets every vertex (and then
every end) of $G$. In this paper, we study a method for finding
Hamiltonian circles of graphs. We illustrate this by extending
several results on finite graphs to Hamiltonian circles in infinite
graphs. For example, we prove that the prism of every 3-connected
cubic graph has a Hamiltonian circle, extending the result of the
finite case by Paulraja.

\end{minipage}
\end{center}

\section{Introduction}

In this paper we always assume that $G$ is a \emph{locally finite}
graph, that is, all its vertices have finite degree. We follow
Diestel \cite{Di16} in our basic terminology for infinite graphs.

A 1-way infinite path is called a \emph{ray} of $G$, and the subrays
of a ray are its \emph{tails}. Two rays of $G$ are \emph{equivalent}
if for every finite set $S\subseteq V(G)$, there is a component of
$G-S$ containing tails of both rays. We write $R_1\approx_GR_2$ if
$R_1$ and $R_2$ are equivalent in $G$. The corresponding equivalence
classes of rays are the \emph{ends} of $G$. We denote by
$\varOmega(G)$ the set of ends of $G$. Let $\alpha\in\varOmega(G)$
and $S\subseteq V(G)$ be a finite set. We denote by $C(S,\alpha)$
the unique component of $G-S$ that containing a ray (and a tail of
every ray) in $\alpha$. We let $\varOmega(S,\alpha)$ be the set of
all ends $\beta$ with $C(S,\beta)=C(S,\alpha)$.

To built a topological space $|G|$ we associate each edge $uv\in
E(G)$ with a homeomorphic image of the unit real interval $[0,1]$,
where 0,1 map to $u,v$ and different edges may only intersect at
common endpoints. Basic open neighborhoods of points that are
vertices or inner points of edges are defined in the usual way, that
is, in the topology of the 1-complex. For an end $\alpha$ we let the
basic neighborhood
$\widehat{C}(S,\alpha)=C(S,\alpha)\cup\varOmega(S,\alpha)\cup
E(S,\alpha)$, where $S\subseteq V(G)$ is finite and $E(S,\alpha)$ is
the set of all inner points of the edges between $C(S,\alpha)$ and
$S$. This completes the definition of $|G|$, called the Freudenthal
compactification of $G$. In \cite{Di16} it is shown that if $G$ is
connected and locally finite, then $|G|$ is a compact Hausdorff
space.

An \emph{arc} of $G$ is the imagine of a homeomorphic map of the
unit interval $[0,1]$ in $|G|$; and a \emph{circle} is the imagine
of a homeomorphic map of the unit circle $S^1$ in $|G|$. A circle of
$G$ is \emph{Hamiltonian} if it meets every vertex (and then every
end) of $G$.

Diestel \cite{Di10} launched the ambitious project of extending
results on Hamiltonian cycles in finite graphs to Hamiltonian
circles in infinite graphs. He specifically conjectured that the
square of every 2-connected locally finite graph has a Hamiltonian
circle \cite{Di05}, in order to obtain a unification of Fleischner's
theorem \cite{Fl}. This was confirmed by Georgakopoulos \cite{Ge09}.

Georgakopoulos (see \cite{Di10}) then conjectured that the line
graph of every 4-edge-connected graph has a Hamiltonian circle.
Bruhn (see \cite{Di10}) conjectured that Tutte's theorem on
Hamiltonian cycles in 4- connected planar graphs can be extended to
Hamiltonian circles. These conjectures are open but significant
progress has been made by Lehner \cite{Le} on the former and by
Bruhn and Yu \cite{BrYu} on the latter. Several other results in
this area can be found in \cite{CuWaYu,HaLePo,He15,He16}.

In the present paper, we study a method for finding Hamiltonian
circles, which is closely related to the faithful subgraphs, the end
degrees and the Hamiltonian curves (see Section 2). We apply this by
extending several results on finite graphs. Specially we prove that
the prism of every 3-connected cubic graph has a Hamiltonian circle,
extending the finite result by Paulraja \cite{Pa}.

The paper is organized as follows: In Section 2, we give some
additional concepts, following which we present our main result. In
Section 3, we give a proof of our main result. In Section 4, we give
some properties of faithful subgraphs. In Sections 5, 6 and 7, we
extend some results on finite graphs by applying our result.

\section{Main Result}

Before giving our main result, we need some additional terminology.
A subgraph $F$ is called \emph{faithful} to $G$ if\\
(1) every end of $G$ contains a ray of $F$; and\\
(2) for any two rays $R_1,R_2$ of $F$, $R_1\approx_FR_2$ if and only
if $R_1\approx_GR_2$.\\
If $F\leq G$, then for every finite set $S\in V(F)$, each component
of $F-S$ is contained in a component of $G-S$. Thus the condition
(2) can be replaced by `for any two rays $R_1,R_2$ of $F$,
$R_1\approx_GR_2$ implies $R_1\approx_FR_2$'.

The \emph{(vertex-)degree} of an end $\alpha\in\varOmega(G)$ is the
maximum number of vertex-disjoint rays in $\alpha$; and the
edge-degree of $\alpha$ is the maximum number of edge-disjoint rays
in $\alpha$. We denote by $d(\alpha)$ the degree of an end $\alpha$.
We refer the reader to \cite{BrSt} for some properties on the end
degrees of graphs.

We define a \emph{curve} of $G$ as the imagine of a continuous map
of the unit interval $[0,1]$ in $|G|$. A curve is \emph{closed} if
$0,1$ map to the same point; and is \emph{Hamiltonian} if it is
closed and meets every vertex of $G$ exactly once. In other words, a
Hamiltonian curve is the imagine of a continuous map of the unit
circle $S^1$ in $|G|$ that meets every vertex of $G$ exactly once.
Thus a Hamiltonian curve may repeat ends. Note that a Hamiltonian
circle is a Hamiltonian curve but not vice versa. We refer the
reader to \cite{KuLiTh} for some properties on the Hamiltonian
curves.

Now we give a necessary and sufficient condition for the existence
of Hamiltonian circles.

\begin{theorem}\label{ThMain}
For an infinite locally finite graph $G$, the following three
statements are equivalence:\\
(1) $G$ has a Hamiltonian circle.\\
(2) $G$ has a spanning faithful subgraph $F$ with a Hamiltonian
curve and for every $\alpha\in\varOmega(F)$, $d(\alpha)=2$.\\
(3) $G$ has a spanning faithful subgraph $F$ with a Hamiltonian
curve and for every $\alpha\in\varOmega(F)$, $d(\alpha)\leq 3$.
\end{theorem}

The proof is postponed to the next section. We remark that in
\cite{KuLiTh}, the authors gave a characterization of Hamiltonian
circles. Notice that a 2-factor $F$ is a set of edges in $G$ such
that each vertex of $G$ is incident to exactly two edges in $F$.

\begin{theorem}[K\"undgen et al. \cite{KuLiTh}]
Let $G$ be an infinite locally finite graph. Then every Hamiltonian
circle of $G$ corresponds to a 2-factor $F$ of $G$ such that\\
(1) every finite cut intersects $F$ a positive even number of times,
and \\
(2) for each two distinct edges $e_1,e_2\in F$, $G$ has a finite cut
$M$ such that $F\cap M=\{e_1,e_2\}$.\\
Conversely, if a 2-factor $F$ satisfies (1)(2), then the closure of
$F$ is a Hamiltonian circle of $G$.
\end{theorem}

\section{Proof of Theorem \ref{ThMain}}

For a finite graph $G$, if $G$ has a spanning subgraph $F$ that has
a Hamiltonian cycle, then $G$ itself has a Hamiltonian cycle. But
this is not true for Hamiltonian circles. The main reason is that we
have to guarantee injectivity at the ends in Hamiltonian circles. We
first show that the existence of Hamiltonian circles is stable for
faithful spanning subgraphs.

\begin{lemma}\label{LeFaithfulCircle}
Let $G$ be an infinite locally finite graph, and let $F$ be a
faithful spanning subgraph of $G$. If $F$ has a Hamiltonian circle,
then $G$ has a Hamiltonian circle.
\end{lemma}

\begin{proof}
We define a map $\pi: \varOmega(F)\rightarrow\varOmega(G)$ such that
for an end $\alpha\in\varOmega(F)$, $\pi(\alpha)$ is the end of $G$
containing all rays in $\alpha$. By the definition of the faithful
subgraphs, $\pi$ is a bijection between $\varOmega(F)$ and
$\varOmega(G)$ (see \cite{Ge09}). Let $\alpha\in\varOmega(F)$ and
$S\subseteq V(G)$ be finite. Since $F\leq G$, the component
$C(S,\alpha)$ of $F-S$ is contained in $C(S,\pi(\alpha))$. If there
is an end $\beta\in\varOmega(S,\alpha)$, then every ray in $\beta$
has a tail contained in $C(S,\alpha)$, which is contained in
$C(S,\pi(\alpha))$. This implies that
$\pi(\beta)\in\varOmega(S,\pi(\alpha))$. It follows that
$\pi(\varOmega(S,\alpha))\subseteq\varOmega(S,\pi(\alpha))$.

Now let $\sigma_F: S^1\rightarrow|F|$ be a Hamiltonian circle of
$F$. We define $\sigma_G: S^1\rightarrow|G|$ such that
$$\sigma_G(p)=\left\{\begin{array}{ll}
  \pi(\sigma_F(p)), & \mbox{if }\sigma_F(p)\in\varOmega(F);\\
  \sigma_F(p),      & \mbox{otherwise}.
\end{array}\right.$$
Clearly the map $\sigma_G$ is injective and meets all vertices of
$V(G)$. Now we prove that it is continuous.

Since $\sigma_F$ is homeomorphic, $\sigma_G$ is continuous at point
$p$ if $\sigma_F(p)$ is a vertex or is an inner point of an edge.
Now we assume that $\sigma_F(p)=\alpha\in\varOmega(F)$. Let
$\boldsymbol{p}=(p_i)_{i=0}^\infty$ be a sequence of points in $S^1$
converges to $p$ and let $S\subseteq V(G)$ be a finite set. Since
$\sigma_F$ is continuous, the neighborhood $\widehat{C}(S,\alpha)$
of $\alpha$ contains almost all terms of $\boldsymbol{p}$ (that is,
there exists $j$ such that $\sigma(p_i)\in\widehat{C}(S,\alpha)$ for
all $i\geq j$). Recall that $C(S,\alpha)\subseteq C(S,\pi(\alpha))$,
$E(S,\alpha)\subseteq E(S,\pi(\alpha))$ and
$\pi(\varOmega(S,\alpha))\subseteq\varOmega(S,\pi(\alpha))$. It
follows that $\widehat{C}(S,\pi(\alpha))$ contains almost all terms
of $(\sigma_G(p_i))_{i=0}^\infty$, and thus
$(\sigma_G(p_i))_{i=0}^\infty$ converges to $\alpha$. This implies
that $\sigma_G$ is continuous and is a Hamiltonian circle of $G$.
\end{proof}

We will make use of K\"onig's Infinity Lemma.

\begin{lemma}[K\"onig \cite{Ko}, see also \cite{Di16}]\label{LeKo}
Let $V_0,V_1,V_2,\ldots$ be an infinite sequence of disjoint
non-empty finite sets, and let $G$ be a graph on
$\bigcup_{i=0}^\infty V_i$. Assume that every vertex in $V_i$ has a
neighbor in $V_{i-1}$, $i\geq 1$. Then $G$ has a ray
$R=v_0v_1v_2\ldots$ with $v_i\in V_i$ for all $i\geq 0$.
\end{lemma}

Let $G$ be a graph and $S\subseteq V(G)$. We set
$$N_G(S)=\bigcup_{v\in S}N_G(v)\backslash S,\ Z_G(S)=\{v\in S:
N_G(v)\backslash S\neq\emptyset\}, \mbox{ and } I_G(S)=S\backslash
Z_G(S).$$ Note that $Z_G(S)=N_G(V(G)\backslash S)$. For a graph $F$
with $V(F)\subseteq V(G)$, we use $N_G(F)$, $Z_G(F)$, and $I_G(F)$
instead of $N_G(V(F))$, $Z_G(V(F))$, and $I_G(V(F))$, respectively.

\begin{lemma}\label{LeDegreeEnd}
Let $G$ be an infinite locally finite graph and
$\alpha\in\varOmega(G)$. Then the following three statements are
equivalent:\\
(1) $d(\alpha)\leq k$;\\
(2) for every finite $S\subseteq V(G)$, there is a finite
$T\subseteq V(G)$, $S\subseteq T$, such that $|Z_G(C(T,\alpha))|\leq
k$;\\
(3) for every finite $S\subseteq V(G)$, there is a finite
$T\subseteq V(G)$, $S\subseteq T$, such that $|N_G(C(T,\alpha))|\leq
k$.
\end{lemma}

\begin{proof}
(1) $\Rightarrow$ (2). Suppose that there exists a finite set
$S\subseteq V(G)$, such that for every finite set $T\subseteq V(G)$
containing $S$, $|Z_G(C(T,\alpha))|\geq k+1$. We take such an $S$
with $|S|\geq k+1$. Set $S_0=S$, and for $i=1,2,\ldots$, set
$S_i=S_{i-1}\cup Z_G(C(S_{i-1},\alpha))$. By Menger's Theorem, we
can see that $G[S_{i+1}]$ has $k+1$ vertex-disjoint paths between
$S_0$ and $C(S_i,\alpha)$. Let $\mathcal{U}_i$ be the set of the
unions of $k+1$ vertex-disjoint paths between $S_0$ and
$C(S_i,\alpha)$. Since $S_{i+1}$ is finite, we have that
$\mathcal{U}_i$ is finite for every $i\geq 0$.

We define a graph $\mathcal{G}$ on
$\bigcup_{i=0}^\infty\mathcal{U}_i$ such that
$U_{i-1}\in\mathcal{U}_{i-1}$ is adjacent to $U_i\in\mathcal{U}_i$
if and only if the $k+1$ paths of $U_{i-1}$ are the subpaths of the
$k+1$ paths of $U_i$. Clearly every vertex in $\mathcal{U}_i$ has a
neighbor in $\mathcal{U}_{i-1}$. By Lemma \ref{LeKo}, $\mathcal{G}$
has a ray $\mathcal{R}=U_0U_1U_2\ldots$ with $U_i\in\mathcal{U}_i$,
$i\geq 0$. It follows that $\bigcup_{i=0}^\infty U_i$ is the union
of $k+1$ vertex-disjoint rays in $\alpha$, implying that the degree
of $\alpha$ is at least $k+1$, a contradiction.

(2) $\Rightarrow$ (3). This can be deduced by the fact that
$|N_G(C(T,\alpha))|\leq |Z_G(C(T\backslash Z_G(T),\alpha))|$.

(3) $\Rightarrow$ (1). Suppose that $d(\alpha)\geq k+1$. Let
$R_1,\ldots,R_{k+1}$ be $k+1$ rays in $\alpha$, and let $S$ be the
set of the origins of the $k+1$ rays. Let $T$ be an arbitrary finite
set with $S\subseteq T\subseteq V(G)$. It follows that $C(T,\alpha)$
contains a tail of every ray $R_i$, $i=1,\ldots,k+1$. Let $u_i$ be
the first vertex along $R_i$ appearing in $C(T,\alpha)$. Thus
$u_i\in Z_G(C(T,\alpha))$, implying that $|Z_G(C(T,\alpha))|\geq
k+1$, a contradiction.
\end{proof}

\begin{lemma}\label{LeEndDegreeCircle}
If every end of $G$ has degree at most 3, then every Hamiltonian
curve of $G$ is also a Hamiltonian circle.
\end{lemma}

\begin{proof}
Let $C$ be a Hamiltonian curve of $G$. It sufficient to show that
$C$ passes through each end exactly once. Suppose that it passes
through an end $\alpha$ at least twice. Let $S$ be a set of two
vertices that separating two copies of $\alpha$ on $C$. For an
arbitrary finite set $T$ with $S\subseteq T\subseteq V(G)$, $T$
divide $C$ in to $|T|$ intervals (maximal curves of $C$ that is
internally disjoint from $T$). At least two of the intervals
containing a copy of $\alpha$, say $I_1=C[u_1,v_1]$,
$I_2=C[u_2,v_2]$, where $u_1,u_2,v_1,v_2\in T$. Since $|G|$ is
compact, we have that
$I_1\backslash\{u_1,v_1\},I_2\backslash\{u_2,v_2\}$ are both
contained in $\widehat{C}(T,\alpha)$. It follows that
$u_1,u_2,v_1,v_2\in N_G(C(T,\alpha))$, implying that
$|N_G(C(T,\alpha))|\geq 4$. By Lemma \ref{LeDegreeEnd},
$d(\alpha)\geq 4$, a contradiction.
\end{proof}

In \cite{KuLiTh}, the authors obtained some necessary and sufficient
conditions for a graph $G$ to have a Hamiltonian curve. We list one
of the conditions which we will use in our proof.

\begin{theorem}[K\"undgen et al. \cite{KuLiTh}]\label{ThKuLiTh}
An infinite locally finite graph $G$ has a Hamiltonian curve if and
only if every finite set $S\subseteq V(G)$ is contained in a cycle
of $G$.
\end{theorem}

Now we give the proof Theorem \ref{ThMain}.

\begin{proof}
The assertion $(2)\Rightarrow(3)$ is trivial and the assertion
$(3)\Rightarrow(1)$ was deduced by Lemmas \ref{LeFaithfulCircle} and
\ref{LeEndDegreeCircle}. Now we show that $(1)\Rightarrow(2)$. Let
$C$ be a Hamiltonian circle of $G$ (with a given orientation).
Recall that there is a bijection between the end set of $G$ and that
of any faithful subgraph $F$ of $G$. Therefore if $E(C)\subseteq
E(F)$, then we have a Hamiltonian circle of $F$ by using each end of
$F$ instead of the corresponding end of $G$. In this meaning, we may
say $C$ is a Hamiltonian circle of $F$ without ambiguity. Also note
that if we remove finitely many of edges from $G$, then the
resulting graph is faithful to $G$.

For any finite subset $S\subseteq V(G)$, $S$ divides $C$ into $|S|$
arcs, called \emph{$S$-intervals}. An edge $uv\in E(G)$ with
$u,v\notin S$ is \emph{crossed} with $S$ if $u,v$ are contained in
two distinct $S$-intervals. We claim that $G$ has only finitely many
of edges crossed with $S$. If there are infinitely many of edges
between some two $S$-intervals, then some end of $G$ will appear in
both $S$-intervals, a contradiction. Since there are only finitely
many of $S$-intervals, we have that there are only finitely many of
edges crossed with $S$.

Set $V(G)=\{v_1,v_2,\ldots\}$. Let $C_1$ be a cycle of $G$
containing $v_1$, $S_1=V(C_1)$, and $G_1$ be the graph obtained from
$G$ by removing all edges crossed with $S_1$. Clearly $G_1$ is
faithful to $G$ and $C$ is a Hamiltonian circle of $G_1$ as well.
Now for $i=2,3,\ldots$, let $C_i$ be a cycle of $G_{i-1}$ containing
$S_{i-1}\cup\{v_i\}$, $S_i=V(C_i)$ and $G_i$ be the graph obtained
from $G_{i-1}$ by removing all edges crossed with $S_i$. Note that
for each $i\geq 2$, $G_i$ is faithful to $G_{i-1}$ and $C$ is a
Hamiltonian circle of $G_i$. By Theorem \ref{ThKuLiTh}, the cycle
$C_{i+1}$ exists. Moreover, for any edge $uv\in E(G_i)$ with $u,v\in
S_i$, $uv\in E(G_j)$ for all $j\geq i$.

Now let $F$ be the spanning subgraph of $G$ such that for any edge
$v_iv_j\in E(G)$, $v_iv_j\in E(F)$ if and only if $v_iv_j\in
E(G_{\max\{i,j\}})$. It remained to show that $F$ is faithful to
$G$, $F$ has a Hamiltonian curve and every end of $F$ has degree 2.

For any finite set $S\subseteq V(G)$, let $j=\max\{i: v_i\in S\}$.
Then $C_j$ is a cycle of $G_j$ containing $S$, and clearly
$E(C_j)\subseteq E(F)$. By Theorem \ref{ThKuLiTh}, $F$ has a
Hamiltonian curve.

Let $R_1,R_2$ be two rays of $F$ with $R_1\approx_GR_2$, and let
$S\subseteq V(G)$ be an arbitrary finite set. For convenient we
assume $|S|\geq 3$. Let $R'_1,R'_2$ be the tails of $R_1,R_2$
contained in $F-S$. Set $j=\max\{i: v_i\in S\}$ (so $S\subseteq
S_j$). Since $E(G)\backslash E(G_j)$ is finite, $G_j$ is faithful to
$G$. Recall that $C$ is a Hamiltonian circle of $G_j$. Let
$\alpha\in\varOmega(G)$, $\alpha_j\in\varOmega(G_j)$ be such that
$\{R_1,R_2\}\subseteq\alpha_j\subseteq\alpha$. It follows that there
is a unique $S_j$-intervals, say $I_j=C[a,b]$, that passes through
$\alpha_j$. Let $u_1\in V(R'_1),u_2\in V(R'_2)$, be two vertices
contained in $I_j\backslash\{a,b\}$. By the definition of $G_j$,
$\{a,b\}$ is a cut of $G_j$ separating $\{u_1,u_2\}$ and
$S_j\backslash\{a,b\}$. Let $k$ be such that
$S_j\cup\{u_1,u_2\}\subseteq S_k$. Thus $C_k$ is a cycle of $G_k$
containing $S_j\cup\{u_1,u_2\}$. Now $C_k$ contains a
$(u_1,u_2)$-path that vertex-disjoint with $S_j$. Since
$E(C_k)\subseteq E(F)$, $u_1,u_2$ are connected in $F-S_j$, and
then, in $F-S$. It follows that $R_1\approx_FR_2$. Thus $F$ is
faithful to $G$.

Let $\alpha_F\in\varOmega(F)$ and $S\subseteq V(G)$ be an arbitrary
finite set. Since $F$ has a Hamiltonian curve, we have
$d(\alpha_F)\geq 2$. Let $\alpha\in\varOmega(G)$ with
$\alpha_F\subseteq\alpha$, and let $S_j,G_j,\alpha_j$ and $I_j$ be
as above. Similarly as above, we can prove that for any two vertices
$u_1,u_2$ contained in $I_j\backslash\{a,b\}$, $u_1,u_2$ are
connected in $F-S_j$. It follows that
$C(S_j,\alpha_F)=F[V(I)\backslash\{a,b\}]$. Note that
$N_F(V(I)\backslash\{a,b\})=\{a,b\}$. By Lemma \ref{LeDegreeEnd}, we
have $d(\alpha_F)\leq 2$.
\end{proof}

\section{Properties of faithful subgraphs}

\begin{lemma}\label{LeFaithfulFaithful}
Suppose that $D\leq F\leq G$. Then any two of the following
statements imply the third one: (1) $D$ is faithful to $F$; (2) $F$
is faithful to $G$; (3) $D$ is faithful to $G$.
\end{lemma}

\begin{proof}
(1)(2) $\Rightarrow$ (3). Let $\alpha^G$ be an arbitrary end of $G$.
Since $F$ is faithful to $G$, $F$ has a ray $R^F\in\alpha^G$. Let
$\alpha^F$ be the end of $F$ with $R^F\in\alpha^F$. Since $D$ is
faithful to $F$, $D$ has a ray $R^D\in\alpha^F$. Thus
$R^D\approx_FR^F$, implying that $R^D\approx_GR^F$, that is,
$R^D\in\alpha^G$.

Now let $R^D_1,R^D_2$ be two rays of $D$ with $R^D_1\approx_GR^D_2$.
Therefore $R^D_1,R^D_2$ are rays of $F$ as well. Since $F$ is
faithful to $G$, $R^D_1\approx_FR^D_2$. Since $D$ is faithful to
$F$, $R^D_1\approx_DR^D_2$. It follows that $D$ is faithful to $G$.

(1)(3) $\Rightarrow$ (2). Let $\alpha^G$ be an end of $G$. Since $D$
is faithful to $G$, $D$ has a ray $R^D$ contained in $\alpha^G$,
which is also a ray of $F$ since $D\leq F$.

Now let $R^F_1,R^F_2$ be two rays of $F$ with $R^F_1\approx_GR^F_2$.
Let $\alpha^F_i$ be the end of $F$ containing $R^F_i$, and let
$R^D_i$ be a ray of $D$ contained in $\alpha^F_i$, $i=1,2$. It
follows that $R^F_i\approx_FR^D_i$. Since $F\leq G$,
$R^F_i\approx_GR^D_i$. Recall that $R^F_1\approx_GR^F_2$. We have
$R^D_1\approx_GR^D_2$. Since $D$ is faithful to $G$,
$R^D_1\approx_DR^D_2$. Since $D$ is faithful to $F$,
$R^D_1\approx_FR^D_2$. This implies that $R^F_1\approx_FR^F_2$. Thus
$F$ is faithful to $G$.

(2)(3) $\Rightarrow$ (1). Let $\alpha^F$ be an end of $F$. Since $F$
is faithful to $G$, there is an end $\alpha^G$ of $G$ that including
$\alpha^F$. Since $D$ is faithful to $G$, $D$ has an ray $R^D$
contained in $\alpha^G$. Note that $D\leq F$, $R^D$ is also a ray of
$F$. Let $R^F$ be any ray in $\alpha^F$. We have $R^F\in\alpha^G$,
and thus $R^D\approx_GR^F$. Since $F$ is faithful to $G$,
$R^D\approx_FR^F$, and thus $R^D\in\alpha^F$.

Now let $R^D_1,R^D_2$ be two rays of $D$ with $R^D_1\approx_FR^D_2$.
Since $F$ is faithful to $G$, $R^D_1\approx_GR^D_2$. Since $D$ is
faithful to $G$, $R^D_1\approx_DR^D_2$. Thus $D$ is faithful to $F$.
\end{proof}

A \emph{comb} of $G$ is the union of a ray $R$ with infinitely many
disjoint finite paths having precisely their first vertex on $R$;
the last vertices of the paths are the \emph{teeth} of the comb; and
the ray $R$ is the \emph{spine} of the comb. We will use the
following Star-Comb Lemma in our paper.

\begin{lemma}[Diestel \cite{Di16}]\label{LeComb}
If $U$ is an infinite set of vertices in a connected graph, then the
graph contains either a comb with all teeth in $U$ or a subdivision
of an infinite star with all leaves in $U$.
\end{lemma}

Since a locally finite graph $G$ contains no infinite stars, Lemma
\ref{LeComb} always yields a comb of $G$. If $F$ is a connected
spanning subgraph of $G$, then for every ray $R$ of $G$, the spine
$R'$ of a comb of $F$ with all teeth in $V(R)$ is a ray in $F$ with
$R\approx_GR'$. Therefore a connected spanning subgraph $F$ is
faithful to $G$ if and only if for any two rays $R_1,R_2$ of $F$,
$R_1\approx_GR_2$ implies $R_1\approx_FR_2$.

\begin{lemma}\label{LeFaithfulSpanning}
Suppose that $D$ is a connected spanning subgraph of $G$, then for
any graph $F$ with $D\leq F\leq G$, $D$ is faithful to $F$ and $F$
is faithful to $G$.
\end{lemma}

\begin{proof}
Let $R^F_1,R^F_2$ be two rays of $F$ with $R^F_1\approx_GR^F_2$. Let
$\alpha^G$ be the end of $G$ containing $R^F_1,R^F_2$, let
$R^D\in\alpha^G$ be a ray of $D$. By Lemma \ref{LeComb}, $D$ has a
comb with all teeth in $V(R^F_1)$. Let $R^D_1$ be the spine of the
comb. Thus $R^D_1$ is a ray of $D$ and $R^F_1\approx_FR^D_1$. Since
$F\leq G$, $R^F_1\approx_GR^D_1$ and then $R^D\approx_GR^D_1$. Since
$D$ is faithful to $G$, $R^D\approx_DR^D_1$. Since $D\leq F$,
$R^D\approx_FR^D_1$, and then $R^D\approx_FR^F_1$. By a similar
analysis, we have $R^D\approx_FR^F_2$, and thus
$R^F_1\approx_FR^F_2$. This implies that $F$ is faithful to $G$.

Now let $R^D_1,R^D_2$ be two rays of $D$ with $R^D_1\approx_FR^D_2$.
Since $F$ is faithful to $G$, $R^D_1\approx_GR^D_2$. Since $D$ is
faithful to $G$, $R^D_1\approx_DR^D_2$. It follows that $D$ is
faithful to $F$.
\end{proof}

\section{Prisms of 3-connected cubic graphs.}

The prism of a graph $G$ is the Cartesian product $G\square K_2$.
Prisms over 3-connected planar graphs are examples of 4-polytopes.
In 1973, Rosenfeld and Barnette \cite{RoBa} showed that every cubic
planar 3-connected graph has a Hamiltonian prism, under the
assumption of the Four Color Conjecture, which is open at that time.
Fleischner \cite{Fl89} found in 1989 a proof avoiding the the Four
Color Theorem. Eventually, Paulraja \cite{Pa} showed that planarity
is inessential here.

\begin{theorem}[Paulraja \cite{Pa}]\label{ThPa}
If $G$ is a finite 3-connected cubic graph, then $G\square K_2$ is
Hamiltonian.
\end{theorem}

As an application of our main result, we show that Paulraja's
theorem can be extended to infinite graphs.

\begin{theorem}\label{ThCubic}
If $G$ is an infinite 3-connected cubic graph, then $G\square K_2$
has a Hamiltonian circle.
\end{theorem}

\subsection{From 3-connected cubic graphs to 2-connected bipartite graphs}

In this subsection, we will show that every 3-connected cubic graph
has a faithful spanning 2-connected bipartite subgraphs.

Let $G$ be a graph and $S\subseteq V(G)$. We say that $S$ is
$k$-connected in $G$ if each two vertices in $S$ are connected by
$k$ internally disjoint paths of $G$; and $S$ is $k$-edge-connected
in $G$ if each two vertices in $S$ are connected by $k$
edge-disjoint paths of $G$. We remark that for the case of $G$ being
subcubic, $S$ is $k$-connected in $G$ if and only if $S$ is
$k$-edge-connected in $G$, for $k=1,2,3$.

\begin{lemma}\label{LeCoBiSu}
Let $G$ be a subcubic graph, $S$ be a finite subset of $V(G)$ and
$D$ be a finite 2-connected bipartite subgraph of $G$. If $S\cap
V(D)\neq\emptyset$ and $S$ is 3-connected in $G$, then $G$ has a
finite 2-connected bipartite subgraph $F$ with $D\leq F$ and
$S\subseteq V(F)$.
\end{lemma}

\begin{proof}
We let $F$ be a finite 2-connected bipartite subgraph of $G$ with
$D\leq F$ and $F$ contains vertices of $S$ as many as possible. We
will show that $F$ contains all vertices in $S$. Since $S\cap
V(D)\neq\emptyset$, we have $S\cap V(F)\neq\emptyset$. Assume that
$S\backslash V(F)\neq\emptyset$. Let $u\in S\cap V(F)$ and $v\in
S\backslash V(F)$. Since $S$ is 3-connected in $G$, there are 3
internally disjoint paths between $u$ and $v$. It follows that there
are three paths $P_1,P_2,P_3$ from $v$ to $F$ such that they have
the only vertex $v$ in common. Thus we can add two of the paths to
$F$ to obtain a 2-connected bipartite graph containing more vertex
in $S$ than $F$, a contradiction.
\end{proof}

\begin{lemma}\label{LeComponentBipartite}
Let $G$ be a subcubic graph, and $\mathcal{U}$ be a finite class of
finite subsets of $V(G)$. Suppose that each $U\in\mathcal{U}$ is
3-connected in $G$. Then $G$ has a finite subgraph $F$ such that\\
(1) every subset $U\in\mathcal{U}$ is contained in one component of
$F$; and \\
(2) every component of $F$ is either an isolated vertex or a
2-connected bipartite graph.
\end{lemma}

\begin{proof}
We first deal with the case that $G$ is finite. Let $G$ be a
counterexample as small as possible. If $G$ has no even cycle, then
each two vertices of $G$ are not 3-connected in $G$, implying that
all subset $U\in\mathcal{U}$ are singleton. Thus we can take $F$ as
an empty graph on $V(G)$. Now we assume that $G$ contains an even
cycle.

Let $D$ be a 2-connected bipartite subgraph of $G$ with order as
large as possible. By Lemma \ref{LeCoBiSu}, for every
$U\in\mathcal{U}$, if $U\cap V(D)\neq\emptyset$, then $U\subseteq
V(D)$. Let $H$ be an arbitrary component of $G-D$.

Since $D$ is 2-connected, each vertex of $D$ has at most one
neighbor in $H$. If $|E_G(H,D)|\geq 3$, then there are three paths
from some vertex $v\in V(H)$ to $D$ such that they have the only
vertex $v$ in common. Thus we can add two of the paths to $D$ to
obtain a 2-connected bipartite subgraph of $G$ larger than $D$, a
contradiction. Thus we conclude that $|E_G(H,D)|\leq 2$. Specially,
any set $U\in\mathcal{U}$ cannot contain vertices from distinct
components of $G-D$.

Set $\mathcal{U}_H=\{U\in\mathcal{U}: U\subseteq V(H)\}$. If
$|E_G(H,D)|\leq 1$, then let $H'=H$; if $|E_G(H,D)|=2$, saying
$E_G(H,D)=\{uu',vv'\}$ with $u,v\in V(H)$ and $u',v'\in V(D)$, then
let $P$ be a path of $D$ from $u'$ to $v'$, and $H'=H\cup
P\cup\{uu',vv'\}$. It follows that every subset in $\mathcal{U}_H$
is 3-connected in $H'$. By the minimality of $G$, $H'$ has a
subgraph $F_H$ such that every subset in $\mathcal{U}_H$ is
contained in one component of $F_H$, and every component of $F_H$ is
either an isolated vertex or a 2-connected bipartite graph. If
$|E_G(H,D)|=2$ and $P$ is contained in an even cycle of $H'$, then
we can get a 2-connected bipartite subgraph of $G$ larger than $D$,
a contradiction. This implies that $F_H$ is contained in $H$ (with
possibly some isolated vertices in $V(P)$). Now
$$F=D\cup\bigcup\{F_H: H \mbox{ is a component of } G-D\}$$
is a subgraph of $G$ satisfying the requirement.

Now we consider the case that $G$ is infinite. For each two vertices
$u,v$ that contained in a common subset in $\mathcal{U}$, we choose
three paths $P_1^{uv},P_2^{uv},P_3^{uv}$ of $G$ connecting $u$ and
$v$. Let $G'$ be the graph consists of all vertices in
$\bigcup\mathcal{U}$ and all the chosen paths as above. Then $G'$ is
finite and each $U\in\mathcal{U}$ is 3-connected in $G'$ as well.
Thus $G'$ contains a subgraph $F$ satisfying the requirement.
\end{proof}

Let $\mathcal{U}$ be a partition of the vertex set of a graph $G$.
The \emph{quotient graph} $G/\mathcal{U}$ is the multi-graph on
$\mathcal{U}$ such that for each $U_1,U_2\in\mathcal{U}$, $U_1U_2\in
E(G/\mathcal{U})$ if and only if $E_G(U_1,U_2)\neq\emptyset$, and
the multiplicity of $U_1U_2$ is the number of edges in
$E_G(U_1,U_2)$. We call a subdivision of a $K_{1,3}$ a
\emph{$Y$-graph}; and a \emph{$\varTheta$-graph} is a 2-connected
(multi-)graph with two vertices of degree 3 and all other vertices
of degree 2.

\begin{lemma}\label{LeConnectedNeighborhood}
Let $G$ be a 3-connected cubic graph, and $S$ be a finite subset of
$V(G)$. Then there is a finite subset $T$ of $V(G)$ such that $S\cup
N_G(S)\subseteq T$ and for every component $H$ of $G-T$, $N_G(H)$ is
3-connected in $G-S$.
\end{lemma}

\begin{proof}
Since $G-S$ is subcubic, two vertices are connected by 3
edge-disjoint paths if and only if they are connected by 3
internally-disjoint paths. We define an equivalence relation on
$V(G)\backslash S$ such that for any two vertices $u,v\in
V(G)\backslash S$, $u\sim v$ if and only if $u,v$ are 3-connected in
$G-S$ (i.e., $u,v$ can not be separated by an edge-cut of size at
most 2). Let $\mathcal{U}$ be the quotient set of $V(G)\backslash S$
by the equivalence relation.

We claim that for each two equivalence classes
$U_1,U_2\in\mathcal{U}$, $G-S$ has an edge-cut of size at most 2
separating $U_1$ and $U_2$. Let $u_1\in U_1,u_2\in U_2$, and let $M$
be an edge-cut of size at most 2 separating $u_1$ and $u_2$. Suppose
that there are $u'_1\in U_1,u'_2\in U_2$ that are not separated by
$M$. Since $u_i,u'_i$ are 3-connected in $G-S$, $i=1,2$, $u_i,u'_i$
are connected in $G-S-M$, implying that $u_1,u_2$ are not separated
by $M$ in $G-S$, a contradiction. Thus $M$ is an edge-cut of $G-S$
separating $U_1$ and $U_2$.

We now show that $\mathcal{U}$ is finite. Let
$\mathcal{G}=(G-S)/\mathcal{U}$ be the quotient graph. We have that
$\mathcal{G}$ contains no $\varTheta$-graph; otherwise the two
equivalence classes corresponding the vertices of degree 3 in the
$\varTheta$-graph cannot be separated by an edge-cut of $G-S$ of
size at most 2. It follows that each block of $\mathcal{G}$ is a
$K_1$, a $K_2$ or a cycle. Specially, the multiplicity of every edge
of $\mathcal{G}$ is at most 2.

Let $\mathcal{H}$ be a component of $\mathcal{G}$. If $\mathcal{H}$
is infinite, then by Lemma \ref{LeComb}, $\mathcal{H}$ has either an
infinite star, or a ray. If $\mathcal{H}$ has an infinite star, say
with center $U_0\in\mathcal{U}$. Then $\mathcal{H}-U_0$ has infinite
number of components. Since $S$ is finite, there is a component
$\mathcal{L}$ of $\mathcal{H}-U_0$ such that $E_G(S,\bigcup_{U\in
V(\mathcal{L})}U)=\emptyset$. It follows that $S$ and $\bigcup_{U\in
V(\mathcal{L})}U$ are separated by an edge-cut of $G$ of size at
most 2, a contradiction. If $\mathcal{G}$ has a ray, say
$\mathcal{R}=U_1U_2\ldots$, then there is a tail $\mathcal{T}$ of
$\mathcal{R}$ such that $E_G(S,\bigcup_{U\in
V(\mathcal{T})}U)=\emptyset$. It follows that $S$ and $\bigcup_{U\in
V(\mathcal{T})}U$ are separated by an edge-cut of $G$ of size at
most 2, also a contradiction. Thus we conclude that $\mathcal{H}$ is
finite. Clearly $\mathcal{G}$ has finite number of components,
implies $\mathcal{G}$ is finite, and so is
$\mathcal{U}=V(\mathcal{G})$.

For every equivalence class $U\in\mathcal{U}$, $U$ has at most 2
neighbors in each equivalence class
$U'\in\mathcal{U}\backslash\{U\}$, and has finitely number of
neighbors in $S$. This implies that $Z_G(U)$ is finite. Now let
$$T=S\cup\bigcup_{U\in\mathcal{U}}Z_G(U).$$ Clearly $N_G(S)\subseteq T$
and every component $H$ of $G-T$ is contained in $I_G(U)$ for some
$U\in\mathcal{U}$, implying that $N_G(H)$ is 3-connected in $G-S$.
\end{proof}

In the following, we write $F\unlhd G$ if $F$ is a spanning subgraph
of $G$; write $F\leq^FG$ if $F$ is a faithful subgraph of $G$; and
write $F\unlhd^FG$ if $F$ is a faithful spanning subgraph of $G$.

\begin{lemma}\label{LeConnectedBipartite}
Let $G$ be a 3-connected cubic graph, and $D$ be a finite subgraph
of $G$ each component of which is either an isolated vertex or a
2-connected bipartite graph. Then $G$ has a finite 2-connected
bipartite subgraph $F$ with $D\leq F$.
\end{lemma}

\begin{proof}
We first deal with the case that for every component $H$ of $G-D$,
$N_G(H)$ is contained in a non-trivial component of $D$. For this
case we will show that $G$ has a finite 2-connected bipartite
subgraph $F$ with $D\unlhd F$.

Suppose that the assertion is not true, and that $D$ is a
counterexample with smallest number of components. If $D$ has only
one component, then $F=D$ satisfies the requirement. So we assume
that $D$ has at least two components.

Let $\mathcal{L}$ be the set of components of $D$,
$\mathcal{U}=\{V(L): L\in\mathcal{L}\}$, and
$\mathcal{G}=G[V(D)]/\mathcal{U}$. We have that $\mathcal{G}$ is
3-edge-connected; otherwise and edge-cut of $\mathcal{G}$ of size at
most 2 corresponding to an edge-cut of $G$. It follows that
$\mathcal{G}$ contains a $\varTheta$-graph $\mathcal{H}$.

For every vertex $U=V(L)$ of $\mathcal{H}$: if
$d_{\mathcal{H}}(U)=2$, then we change it with a path of $L$; if
$d_{\mathcal{H}}(U)=3$, then we change it with a $Y$-graph of $L$.
Then we get a $\varTheta$-graph $H$ of $G$. Let $C$ be an even cycle
of $H$. It follows that if $C$ passes through some component $L$ of
$D$, then $C$ passes through $L$ exactly once (that is, $C\cap L$ is
a path).

Let $D'=D\cup C$. Then $D'$ has less component number than $D$.
Since $C$ is an even cycle, we see that $D'$ is bipartite. Note that
the only new component of $D'$ is 2-connected. It follows that $D'$
is contained in a 2-connected bipartite graph $F$ of $G$, and
$D\unlhd F$.

Now we consider the general case. By Lemma
\ref{LeConnectedNeighborhood}, there is a finite subset $T\subseteq
V(G)$ such that $V(D)\cup N_G(D)\subseteq T$ and for every component
$H$ of $G-T$, $N_G(H)$ is 3-connected in $G-D$. It follows that
$T\backslash V(D)$ has a partition $\mathcal{U}$ such that each
$U\in\mathcal{U}$ is 3-connected in $G-D$ and for each component $H$
of $G-T$, $N_G(H)$ is contained in some $U\in\mathcal{U}$. By Lemma
\ref{LeComponentBipartite}, $G-D$ has a subgraph $D'$ such that
every subset $U\in\mathcal{U}$ is contained in a component of $D'$
and every component of $D'$ is either an isolated vertex or a
2-connected bipartite graphs. Now $D\cup D'$ is a subgraph of $G$
such that every component of $D\cup D'$ is either an isolated vertex
or a 2-connected bipartite graph, and for every component $H$ of
$G-(D\cup D')$, $N_G(H)$ is contained in a component of $D\cup D'$.
By the analysis above, we can find a finite 2-connected bipartite
subgraph $F$ of $G$ such that $D\leq F$.
\end{proof}

\begin{lemma}\label{LeFaithfulSubgraph1}
Let $G$ be an infinite locally finite connected graph and
$\mathcal{F}=(F_i)_{i=1}^\infty$ be a sequence of finite connected
subgraphs of $G$ such that $F_i\leq F_{i+1}$ and $N_G(F_i)\subseteq
V(F_{i+1})$. Set $F=\bigcup_{i=1}^\infty F_i$. Suppose that for
every component $H$ of $G-F_{i+1}$, there is a component $D$ of
$F_{i+1}-F_i$ such that $N_G(H)\subseteq V(D)$. Then $F\unlhd^FG$.
\end{lemma}

\begin{proof}
Clearly $\bigcup_{i+1}^\infty V(F_i)=V(G)$ and thus $F\unlhd G$.
Since each $F_i$ is connected, we see that $F$ is connected. Suppose
that $F$ is not faithful to $G$. Let $R_1,R_2$ be two rays of $F$
with $R_1\not\approx_F R_2$ and $R_1\approx_G R_2$. Then there is a
finite set $S\subseteq V(G)$ such that $R_1$ and $R_2$ have tails
contained in distinct component of $F-S$ (we take $S$ that contains
the origins of $R_1$ and $R_2$). Let $F_i$ be a graph in
$\mathcal{F}$ such that $S\subseteq V(F_i)$. For $j=1,2$, let $u_j$
be the last vertices in $R_j$ that contained in $S_{i+1}$, $u^+_j$
be the successor of $u_j$ on $R_j$, and $R'_j$ be the tail of $R_j$
with origin $u^+_j$. It follows that $R'_j$ is contained in
$G-F_{i+1}$. Since $R_1\approx_G R_2$, $R'_1$ and $R'_2$ are
contained in a common component $H$ of $G-F_{i+1}$. By assumption,
there is a component $D$ of $F_{i+1}-F_i$ with $N_G(H)\subseteq
V(D)$. It follows that $u_1,u_2\in V(D)$ and thus the component of
$F-F_i$ containing $D$ contains both $R'_1$ and $R'_2$, a
contradiction.
\end{proof}

\begin{theorem}\label{LeFaithfulBipartite}
Every infinite 3-connected cubic graph has a faithful spanning
2-connected bipartite subgraph.
\end{theorem}

\begin{proof}
Let $G$ be an infinite 3-connected cubic graph. We construct a
sequence $\mathcal{F}=(F_i)_{i=1}^\infty$ of finite 2-connected
bipartite subgraph of $G$ such that for $i\geq 1$,\\
(1) $F_i\leq F_{i+1}$ and $N_G(F_i)\subseteq V(F_{i+1})$;\\
(2) for every component $H$ of $G-F_{i+1}$, there is a component $D$
of $F_{i+1}-F_i$ such that $N_G(H)\subseteq V(D)$.

Let $F_1$ be an even cycle of $G$. Suppose we already have $F_i$,
$i\geq 1$. By Lemma \ref{LeConnectedNeighborhood}, there is a finite
$T\subseteq V(G)$ such that $V(F_i)\cup N_G(F_i)\subseteq T$ and for
every component $H$ of $G-T$, $N_G(H)$ is 3-connected in $G-F_i$. By
Lemma \ref{LeComponentBipartite}, $G-F_i$ has a finite subgraph
$D_i$ such that every component of $D_i$ is either an isolated
vertex or a 2-connected bipartite graph, and for every component $H$
of $G-F_i-D_i$, $N_G(H)$ is contained in a component of $D_i$. By
adding isolated vertices to $D_i$, we can take $D_i$ such that
$T\backslash V(F_i)\subseteq V(D_i)$. By Lemma
\ref{LeConnectedBipartite}, $G$ has a 2-connected bipartite subgraph
$F_{i+1}$ with $F_i\cup D_i\unlhd F_{i+1}$. It follows that $F_i\leq
F_{i+1}$, $N_G(F_i)\subseteq V(F_{i+1})$, and for every component
$H$ of $G-F_{i+1}$, there is a component $D$ of $F_{i+1}-F_i$ such
that $N_G(H)\subseteq V(D)$.

By Lemma \ref{LeFaithfulSubgraph1}, $F=\bigcup_{i=1}^\infty F_i$ is
a faithful spanning 2-connected bipartite subgraph of $G$.
\end{proof}

\subsection{From 2-connected subcubic graphs to cacti}

A (finite or infinite) \emph{cactus} is a subcubic graph $G$
consists of a class $\mathcal{C}$ of cycles and
a class $\mathcal{P}$ of paths, such that\\
(1) each two cycles in $\mathcal{C}$ are vertex-disjoint;\\
(2) each two paths in $\mathcal{P}$ are vertex-disjoint; and\\
(3) the graph obtained from $G$ by contracting all cycles in
$\mathcal{C}$ is a tree. \\
The cactus is \emph{even} if each cycle in $\mathcal{C}$ is even. We
call a vertex $v$ of $G$ a \emph{c-vertex} if it is a cut-vertex of
$G$; and a \emph{d-vertex} otherwise. Clearly a d-vertex of a
nontrivial cactus is either of degree 1 or of degree 2 and is
contained in a cycle. We notice that if $G_1,G_2$ are two disjoint
cacti and $v_1,v_2$ are d-vertices of $G_1,G_2$, respectively, then
the graph obtained from $G_1\cup G_2$ by adding an edge $v_1v_2$ is
a cactus.

\begin{lemma}[\v{C}ada et al. \cite{CaKaRoRy}]\label{LeCaKaRoRy}
If $G$ is a finite even cactus, then $G\square K_2$ is Hamiltonian.
\end{lemma}

Let $G$ be a block-chain, and $u,v$ be two vertices of $G$. We say
$G$ is a block-chain \emph{connecting $u,v$} if $G$ is
non-separable, or $G$ is separable and $u,v$ are two inner-vertices
of the two distinct end-blocks of $G$.

\begin{lemma}\label{LeBlockChain}
Let $G$ be a subcubic block-chain connecting $u,v$ and $x$ be a
vertex of $G$. Then $G$ has a finite cactus $F$ with $u,v,x\in V(F)$ such that\\
(1) $u,v$ are d-vertices of $F$; and\\
(2) for every component $H$ of $G-F$, $N_G(H)$ is contained in a
cycle of $F$.
\end{lemma}

\begin{proof}
The assertion is trivial if $G$ has only two vertices. So we assume
that $|V(G)|\geq 3$. Since $G$ is a block-chain connecting $u,v$,
$G$ has a path $P$ connecting $u,v$ and passing through $x$. Suppose
the assertion is not true. We take a counterexample such that the
path $P$ is as short as possible.

Suppose first that $G$ is separable. Let $B$ be an end-block of $G$
such that $x\notin I_G(B)$. Assume without loss of generality that
$u\in I_G(B)$ and let $u'$ be the cut-vertex of $G$ contained in
$B$. It follows that $u'\in V(P)$. If $B$ is 2-connected, then let
$C$ be a cycle of $B$ containing $u,u'$; otherwise $|V(B)|=2$, let
$C=B$. Set $G'=G-I_G(B)$ and $P'=P[u',v]$. Then $G'$ is a
block-chain connecting $u',v$ and $P'$ is a path connecting $u',v$
passing through $x$ that is shorter than $P$. It follows that $G'$
has a cactus $F'$ such that $u',v$ are d-vertices, and for every
component $H$ of $G'-F'$, $N_{G'}(H)$ is contained in a cycle of
$F'$. It follows that $F=C\cup F'$ is a cactus of $G$ satisfying the
requirement.

Suppose now that $G$ is 2-connected. If $G$ has a cycle $C$ with
$u,v,x\in V(C)$, then $C$ is a cactus of $G$ satisfying the
requirement. So we assume that $u,v,x$ are not contained in any
cycles of $G$. Let $C$ be a cycle containing $u,v$, let $u'$ be the
first vertex on $P[x,u]$ that contained in $C$, and $v'$ be the
first vertex on $P[x,v]$ that contained in $C$ (possibly $u=u'$ or
$v=v'$ or both). We take the cycle $C$ such that $P[u',v']$ is as
short as possible. Let $u''$ be the successor of $u'$, and $v''$ be
the predecessor of $v'$ on $P$.

We show that $\{u',v'\}$ is a cut of $G$ separating $x$ and
$V(C)\backslash\{u',v'\}$. Otherwise there is a path $P'$ between
$P[u'',v'']$ and $C-\{u',v'\}$. Let $w$ be the end-vertex of $P'$ on
$C$. It follows that two of the three vertices $u',v',w$ are
contained in a common segments $\overrightarrow{C}[u,v]$ or
$\overrightarrow{C}[v,u]$. Thus there is a cycle $C'$ with $u,v\in
V(C)$ that contains some vertices appear before $u',v'$ on $P[x,u]$
and $P[x,v]$, a contradiction. This implies that $\{u',v'\}$ is a
cut of $G$ separating $x$ and $V(C)\backslash\{u',v'\}$.

Let $G'$ be the component of $G-\{u',v'\}$ containing $x$. Since $G$
is subcubic, we can see that $N_{G'}(u')=\{u''\}$ and
$N_{G'}(v')=\{v''\}$. Let $P'=P[u'',v'']$. Then $G'$ is a
block-chain connecting $u'',v''$ and $P'$ is a path connecting
$u'',v''$ passing through $x$ that is shorter than $P$. It follows
that $G'$ has a cactus $F'$ such that $u'',v''$ are d-vertices of
$F'$, and for every component $H$ of $G'-F'$, $N_{G'}(H)$ is
contained in a cycle of $F'$. It follows that $F=C\cup
F'\cup\{u'u''\}$ is a cactus of $G$ satisfying the requirement.
\end{proof}

\begin{lemma}\label{LeCycleNeigborhood}
Let $G$ be a 2-connected subcubic graph, $C_0$ be a cycle, and $x$
be a vertex of $G$. Then $G$ has a finite cactus $F$ containing
$C_0$ and $x$ such that for every component $H$ of $G-F$, there is a
cycle $C$ of $F$ other than $C_0$ such that $G[V(H)\cup V(C)]$ is
2-connected.
\end{lemma}

\begin{proof}
Suppose the assertion is not true. We take a counterexample such
that $N_G(C_0)$ is as small as possible.

Suppose first that $G-C_0$ has at least two components. Let
$\mathcal{D}=\{D_1,D_2,\ldots,D_k\}$ be the set of components of
$G-C_0$. Since $G$ is subcubic, any two distinct components in
$\mathcal{D}$ have disjoint neighborhood in $C_0$. For every
component $D_i$, let $G_i=G[V(D_i)\cup V(C_0)]$, and let $x_i$ be a
vertex of $G_i$ such that if $x\in V(G_i)$ then $x_i=x$. Now $G_i$
is 2-connected and $N_{G_i}(C_0)$ is smaller than $N_G(C_0)$. It
follows $G_i$ has a cactus $F_i$ containing $C_0$ and $x_i$, and for
every component $H$ of $G_i-F_i$, there is a cycle $C$ of $F_i$
other than $C_0$ such that $G_i[V(H)\cup V(C)]$ is 2-connected. Thus
$F=\bigcup_{i=1}^kF_i$ is a cactus of $G$ satisfying the
requirement, a contradiction. So we assume that $G-C_0$ has only one
component. Let $D=G-C_0$.

If $D$ is trivial, then the cactus consists of $C_0$ and an edge in
$E_G(C_0,D)$ satisfying the requirement, a contradiction. So we
assume that $D$ has at least two vertices.

Suppose now that $|N_G(C_0)|=2$, say $E_G(C_0,D)=\{uu',vv'\}$. Since
$G$ is 2-connected, $uu'$ and $vv'$ are nonadjacent, and the graph
$G'=D$ is a block-chain connecting $u'$ and $v'$. If $x\in V(G')$,
then let $x'=x$; otherwise let $x'$ be an arbitrary vertex of $G'$.
By Lemma \ref{LeBlockChain}, $G'$ has a cactus $F'$ containing $x'$
such that $u',v'$ are d-vertices of $F'$, and for every component
$H$ of $G'-F'$, $N_{G'}(H)$ is contained in a cycle of $F'$. Thus
$F=C_0\cup F'\cup\{uu'\}$ is a cactus of $G$ satisfying the
requirement, a contradiction. So we assume that $|N_G(C_0)|\geq 3$.

Let $\mathcal{B}$ be the set of 2-connected blocks of $D$. Since $G$
is subcubic, each two blocks in $\mathcal{B}$ are disjoint. Let
$\mathcal{G}$ be the graph obtained from $G$ by contracting each
block in $\mathcal{B}$. We notice that $\mathcal{G}-C_0$ is a tree
each leaf of which has a neighbor in $C_0$. It follows that
$\mathcal{G}$ is 2-connected, and there is a vertex in
$\mathcal{G}-C_0$ that has degree at least 3 in $\mathcal{G}$.
Therefore $\mathcal{G}$ has a path $\mathcal{P}$ between a vertex
$u\in V(C_0)$ and $v\in V(\mathcal{G})\backslash V(C_0)$ such that
$d_{\mathcal{G}}(v)\geq3$ and all internal vertices of $\mathcal{P}$
has degree 2 in $\mathcal{G}$. Thus each internal vertex of
$\mathcal{P}$ is either a vertex of $G$ of degree 2, or obtained by
contracting a block in $\mathcal{B}$ that contains exactly two
cut-vertices of $G-C_0$. Let $G'$ be the subgraph of $G$ induced by
the internal vertices of $\mathcal{P}$ and the vertices of the
blocks that corresponding to an internal vertex of $\mathcal{P}$. It
follows that $G'$ is a block-chain, say connecting $u'$ and $v'$,
where $u'\in N_G(u)$, and $v'\in N_G(v)$. If $x\in V(G')$ then let
$x'=x$, otherwise let $x'$ be an arbitrary vertex of $G'$. By Lemma
\ref{LeBlockChain}, $G'$ has a cactus $F'$ containing $u',v',x'$
such that $u',v'$ are d-vertices of $F'$, and for every component
$H$ of $G'-F'$, $N_{G'}(H)$ is contained in a cycle of $F'$.

Let $G''=G-G'$. Clearly $G''$ is 2-connected, and $N_{G''}(C_0)$ is
smaller than $N_G(C_0)$. If $x\in V(G'')$ then let $x''=x$,
otherwise let $x''$ be an arbitrary of $G''$. It follows that $G''$
has a cactus $F''$ containing $C_0$ and $x''$ such that for every
component $H$ of $G''-F''$, $N_{G''}(H)$ is contained in a cycle of
$F''$. Now $F=F'\cup F''\cup\{uu'\}$ is a cactus of $G$ satisfying
the requirement, a contradiction.
\end{proof}

\begin{lemma}\label{LeFaithfulSubgraph2}
Let $G$ be a connected graph and $\mathcal{F}=(F_i)_{i=1}^\infty$ be
a sequence of finite connected subgraphs of $G$ such that $F_i\leq
F_{i+1}$. Set $F=\bigcup_{i=1}^\infty F_i$. Suppose that $F\unlhd
G$, and for every component $H$ of $G-F_{i+1}$, there is a component
$L$ of $F_{i+1}-F_i$ such that $N_F(H)\subseteq V(L)$ (for $i=1$,
set $L=F_1$). Then $F\unlhd^FG$.
\end{lemma}

\begin{proof}
We take a subsequence $\mathcal{F}'$ of $\mathcal{F}$ as follows:
Let $F'_1=F_1$. Suppose we already have $F'_i=F_r$, and we will get
$F'_{i+1}$. Since $N_G(F_r)$ is finite, there exists $F_s$ such that
$N_G(F_r)\subseteq V(F_s)$; and there is $F_t$ such that
$N_G(F_s)\subseteq V(F_t)$. We let $F'_{i+1}=F_t$.

Let $H$ be an arbitrary component of $G-F_t$. We will show that
$N_G(H)$ is connected in $F_t-F_r$. Let $u,v$ be two vertices in
$N_G(H)$. We have $u,v\in V(F_t)\backslash V(F_s)$. Let $P^u,P^v$ be
two paths of $F_t$ between $u,v$, respectively, to $F_s$. Suppose
that $u',v'$ are the termini of $P^u,P^v$, respectively. Clearly
$P^u-u'$, $P^v-v'$ and $H$ are contained in the same component of
$G-F_s$. By our assumption, $u',v'$ are connected, say by a path
$P$, in $F_s-F_{s-1}$. Now $P^uu'Pv'P^v$ is a path of $F_t-F_r$
connecting $u,v$. It follows that there is a component $L$ of
$F_t-F_r$ such that $N_G(H)\subseteq V(L)$.

Clearly $F=\bigcup_{i=1}^\infty F'_i$ and $N(F'_i)\subseteq
V(F'_{i+1})$. By Lemma \ref{LeFaithfulSubgraph1}, $F\unlhd^FG$.
\end{proof}

\begin{theorem}\label{LeFaithfulCactus}
Every infinite 2-connected subcubic graph has a faithful spanning
cactus.
\end{theorem}

\begin{proof}
Let $G$ be a 2-connected subcubic graph, with
$V(G)=\{x_1,x_2,\ldots\}$. We construct a sequence
$\mathcal{F}=(F_i)_{i=1}^\infty$ of
finite cactus of $G$ such that for $i\geq 1$,\\
(1) $x_i\in V(F_i)$ and $F_i\leq F_{i+1}$;\\
(2) for every component $H$ of $G-F_{i+1}$, there is a cycle $C$ of
$F_{i+1}-F_i$ such that $G[V(H)\cup V(C)]$ is 2-connected; and if an
edge $uv\in E_G(F_{i+1},H)$ is in $E(F_{i+2})$, then $u\in V(C)$.

Let $F_1$ be a cycle containing $x_1$. Suppose we already have
$F_i$. Let $L$ be an arbitrary component of $G-F_i$, and let $C_L$
be a cycle of $F_i-F_{i-1}$ such that $G_L=G[V(L)\cup V(C_L)]$ is
2-connected (if $i=1$, then set $C_L=F_1$). If $x_i\in V(L)$, then
let $x_L=x_i$; otherwise let $x_L$ be an arbitrary vertex of $L$. By
Lemma \ref{LeCycleNeigborhood}, $G_L$ has a finite cactus $F_L$
containing $C_L$ and $x_L$ such that for every component $H$ of
$G_L-F_L$, there is a cycle $C$ of $F_L$ other than $C_L$ such that
$G_L[V(H)\cup V(C)]$ is 2-connected. Now let
$$F_{i+1}=F_i\cup\bigcup\{F_L: L\mbox{ is a component of }G-F_i\}.$$
It follows that $x_i\in V(F_i)$, $F_i\leq F_{i+1}$, and for every
component $H$ of $G-F_{i+1}$, there is a cycle $C$ of $F_{i+1}-F_i$
such that $G[V(H)\cup V(C)]$ is 2-connected. Moreover, by the
construction above, if an edge $uv\in E_G(F_{i+1},H)$ is in
$E(F_{i+2})$, then $u\in V(C)$.

Let $F=\bigcup_{i=1}^\infty F_i$. Clearly $\bigcup_{i=1}^\infty
V(F_i)=V(G)$ and thus $F\unlhd G$. By Lemma
\ref{LeFaithfulSubgraph2}, $F$ is a faithful spanning cactus of $G$.
\end{proof}

\subsection{Proof of Theorem \ref{ThCubic}}

We need some additional lemmas concerning the Cartesian product of
graphs. In the following lemma, we review the graph $G$ as its first
copy in the Cartesian product $G\square D$.

\begin{lemma}\label{LeFaithfulPrism}
Let $G$ be an infinite locally finite graph and $D$ be a finite
connected graph. Then $G\leq^FG\square D$.
\end{lemma}

\begin{proof}
Let $G'=G\square D$. For a vertex $v\in V(G')$, we use $\sigma(v)$
to denote the corresponding vertex of $G$ (i.e., $v$ is a copy of
$\sigma(v)$). Let $\alpha'\in\varOmega(G')$ and $R'=u'_1u'_2\ldots$
be a ray of $G'$ contained in $\alpha'$. We will find a ray $R$ of
$G$ with $R\in\alpha'$. Let $u_1=\sigma(u'_1)$. Suppose we already
define $u_i$. Let $j$ be the maximum integer with
$u_i=\sigma(u'_j)$, and let $u_{i+1}=\sigma(u'_{j+1})$. It is easy
to see that $R=u_1u_2\ldots$ is a ray of $G$ and $R\approx_{G'}R'$,
and thus $R\in\alpha'$.

Now let $R_1,R_2$ be two rays of $G$ with $R_1\approx_{G'}R_2$. We
will show that $R_1\approx_GR_2$. Let $S$ be a finite subset of
$V(G)$, and $S'$ be union of the copies of $S$. Thus $S'\subseteq
V(G')$ is finite, and there is a component $H'$ of $G'-S'$ such that
$R_1,R_2$ have tails in $H'$. It follows that $H'=H\square D$ for
some component $H$ of $G-S$. This implies that $H$ contains tails of
both $R_1,R_2$. Therefore $R_1\approx_GR_2$, and $G\leq^FG'$.
\end{proof}

\begin{lemma}\label{LePrismFaithful}
Let $G$ be an infinite locally finite graph, $F\leq G$, and $D$ be a
finite connected graph. If $F\leq^FG$, then $F\square
D\leq^FG\square D$.
\end{lemma}

\begin{proof}
By Lemma \ref{LeFaithfulPrism}, $F\leq^FF\square D$ and
$G\leq^FG\square D$. By Lemma \ref{LeFaithfulFaithful} and the fact
$F\leq^FG$, $F\leq^FG\square D$. Again by Lemma
\ref{LeFaithfulFaithful}, $F\square D\leq^FG\square D$.
\end{proof}

\begin{lemma}\label{LeDegreeEndPrism}
Let $G$ be an infinite locally finite graph, $D$ be a finite
connected graph, $G'=G\square D$. Let $\alpha$ be an end of $G$ and
$\alpha'$ be an end of $G'$ with $\alpha\subseteq\alpha'$. Then
$d(\alpha')=d(\alpha)|V(D)|$.
\end{lemma}

\begin{proof}
Let $k=d(\alpha)$, $n=|V(D)|$, and let $\mathcal{R}$ be the set of
$k$ vertex-disjoint rays in $\alpha$. It follows that for each ray
$R\in\mathcal{R}$ there are $n$ copies of $R$ in $G'$. Note that the
copes of two vertex-disjoint rays are vertex-disjoint. Thus
$\alpha'$ contains at least $kn$ vertex-disjoint rays of $G'$, i.e.,
$d(\alpha')\geq kn$.

Suppose now that $d(\alpha')>kn$. By Lemma \ref{LeDegreeEnd}, there
is a finite set $S'\subseteq V(G')$ such that for every finite set
$T'\subseteq V(G')$ with $S'\subseteq T'$,
$|Z_{G'}(C(T',\alpha'))|>kn$. Let $S$ be the set of vertices in $G$
that has some copies in $S'$. Let $T\subseteq V(G)$ be a finite set
with $S\subseteq T$, and let $T'$ be the set of vertices that are
copies of vertices in $T$. Thus $|Z_{G'}(C(T',\alpha'))|>kn$.
Clearly $C(T',\alpha')=C(T,\alpha)\square D$. It follows that every
edge in $E(T',\alpha')$ is the copy of an edge in $E(T,\alpha)$.
This implies $|Z_G(C(T,\alpha))|>k$. By Lemma \ref{LeDegreeEnd},
$d(\alpha)>k$, a contradiction.
\end{proof}

\begin{lemma}\label{LeCactusHamiltonDegree}
Let $G$ be an infinite even cactus. Then (1) $G\square K_2$ has a
Hamiltonian curve; and (2) every end of $G\square K_2$ has degree 2.
\end{lemma}

\begin{proof}
(1) Set $G'=G\square K_2$. For an arbitrary finite set $S'$, let $S$
be the set of vertices in $G$ that have some copies in $S'$. It
follows that $G$ has a finite sub-cactus $F$ containing $S$. By
Lemma \ref{LeCaKaRoRy}, $F\square K_2$ is Hamiltonian, implying that
$S'$ is contained in a finite cycle of $G'$. By Theorem
\ref{ThKuLiTh}, $G'$ has a Hamiltonian curve.

(2) For every finite set $S\subseteq V(G)$, $G$ has a finite
sub-cactus $F$ with $S\subseteq V(F)$. Let $D$ be the subgraph of
$G$ induced by $$V(F)\cup\{v\in V(G): v \mbox{ is contained in some
cycle } C \mbox { of } G \mbox{ such that } V(C)\cap
V(F)\neq\emptyset\}.$$ Then for every component $H$ of $G-F$,
$|N_G(H)|=1$. By Lemma \ref{LeDegreeEnd}, every end of $G$ has
degree 1. By Lemma \ref{LeDegreeEndPrism}, every end of $G\square
K_2$ has degree 2.
\end{proof}

Now we prove Theorem \ref{ThCubic}.

\begin{proof}
Let $G$ be an infinite 3-connected cubic graph. By Lemmas
\ref{LeFaithfulFaithful}, \ref{LeFaithfulBipartite} and
\ref{LeFaithfulCactus}, $G$ has a faithful spanning even cactus $F$.
By Lemma \ref{LePrismFaithful}, $F\square K_2\unlhd^FG\square K_2$.

By Lemma \ref{LeCactusHamiltonDegree}, $F\square K_2$ has a
Hamiltonian curve and every end of $F\square K_2$ has degree 2. By
Theorem \ref{ThMain}, $G\square K_2$ has a Hamiltonian circle.
\end{proof}

\section{Prisms of squares of graphs}

In \cite{KaRyKrRoVo}, Kaiser et al. also proved the following
theorems concerning the prisms of squares of graphs and prisms of
line graphs.

\begin{theorem}[Kaiser et al. \cite{KaRyKrRoVo}]\label{ThCaKaRoRy1}
If $G$ is a finite connected graph, then $G^2\square K_2$ is
Hamiltonian.
\end{theorem}

\begin{theorem}[Kaiser et al. \cite{KaRyKrRoVo}]\label{ThCaKaRoRy2}
If $G$ is a finite 2-connected line graph, then $G\square K_2$ is
Hamiltonian.
\end{theorem}

In the following two sections, we will extend the two results to
Hamiltonian circles. Before doing this, we first introduce a concept
semi-cactus.

A \emph{semi-cactus} is a connected graph $G$ such that\\
(1) the maximum degree $\varDelta(G)\leq 4$;\\
(2) every block of $G$ is either a cycle or a $K_2$; and\\
(3) every vertex is contained in one or two blocks of $G$.\\
Similarly as the case of cactus, we can see that every d-vertex of a
semi-cactus $G$ is either of degree 1 or of degree 2 and contained
in a cycle. If $G_1,G_2$ are two disjoint semi-cactus and $v_1,v_2$
are d-vertices of $G_1,G_2$, respectively, then the graph $G$
obtained from $G_1\cup G_2$ by identifying the two vertices
$v_1,v_2$ is also a semi-cactus. The semi-cactus $G$ is \emph{even}
if each cycle of $G$ is even.

Let $v$ be a vertex of $G$. The \emph{cleaving} of $u$ amounts to
replacing the vertex $u$ with two new vertices $u_1,u_2$, making
each edge incident with $u$ incident with exactly one of $u_1,u_2$,
and adding a new edge $e_u=u_1u_2$. Note that cleaving is the
reverse of edge contraction.

\begin{lemma}\label{LeSemiHamilton}
Let $G$ be a finite even semi-cactus, then $G\square K_2$ is
Hamiltonian.
\end{lemma}

\begin{proof}
Let $u$ be an arbitrary vertex with $d_G(u)=4$, let
$N_G(u)=\{v_1,w_1,v_2,w_2\}$ such that $v_i,w_i$ are contained in a
cycle, $i=1,2$. We cleave $u$ such that $u_i$ is adjacent to
$v_i,w_i$ for $i=1,2$. Let $G'$ be the graph obtained from $G$ by
cleaving each vertex of $G$ with degree 4 as above. Then $G'$ is an
even cactus. By Lemma \ref{LeCaKaRoRy}, $G'\square K_2$ has a
Hamiltonian cycle $C'$. For the edge $e_u$ obtained by cleaving a
vertex $u$ with $d_G(u)=4$, the two copies of $e_u$ in $G'\square
K_2$ form an edge-cut of $G'\square K_2$. This implies that $C'$
passing through both copies of $e_u$. Now let $C$ be the cycle
obtained from $C'$ by contracting the copies of each added edge
$e_u$ with $d_G(u)=4$. It follows that $C$ is a Hamiltonian cycle of
$G\square K_2$.
\end{proof}

Similarly as Lemma \ref{LeCactusHamiltonDegree}, we have the
following properties of semi-cacti.

\begin{lemma}\label{LeSemiCactusHamiltonDegree}
Let $G$ be an infinite even semi-cactus. Then (1) $G\square K_2$ has
a Hamiltonian curve; and (2) every end of $G\square K_2$ has degree
2.
\end{lemma}

A rooted spanning tree $T$ of a graph $G$ is \emph{normal} if the
two vertices of every edge in $E(G)$ are comparable in the
tree-order of $T$. The \emph{normal rays} of $T$ are those starting
at the root of $T$. From the following lemma, one can see that a
normal spanning tree of $G$ is faithful to $G$.

\begin{lemma}[Diestel \cite{Di16}]\label{LeNormalFaithful}
If $T$ is a normal spanning tree of $G$, then every end of $G$
contains exactly one normal ray of $T$.
\end{lemma}

One can see that the normal spanning tree has a nice property for
the infinite graphs. From the following theorem, we can always find
a normal spanning tree in infinite locally finite connected graphs.

\begin{theorem}[Jung \cite{Ju}]\label{ThJu}
Every infinite countable connected graph has a normal spanning tree.
\end{theorem}

Recall that the $k$-th power of a graph $G$ is the graph obtained
from $G$ by adding an edge between each two vertices with distance
at most $k$ in $G$.

\begin{lemma}\label{LeFaithfulPower}
For any infinite locally finite connected graph $G$ and an integer
$k\geq 1$, $G\unlhd^FG^k$.
\end{lemma}

\begin{proof}
Clearly $G\unlhd G^k$. Suppose that $R_1,R_2$ are two rays of $G$
with $R_1\approx_{G^k}R_2$. Let $S\subseteq V(G)$ be an arbitrary
finite set, and set $S'=S\cup N_{G^k}(S)$. Clearly $S'$ is finite,
and thus there is a component $C'$ of $G^k-S'$ that contains tails
of both $R_1$ and $R_2$. For any two vertices $u,v\in V(G)\backslash
S'$ with $uv\in E(G^k)$, $G$ has a path $P$ connecting $u,v$ of
length at most $k$. Since both $u,v$ have distance more than $k$
from $S$, $V(P)\cap S=\emptyset$. It follows that $u,v$ are
connected in $G-S$. This implies that all vertices in $V(C')$ are
contained in a common component $C$ of $G-S$. Thus $C$ contains
tails of both $R_1$ and $R_2$, implying that $G\unlhd^FG^k$.
\end{proof}

\begin{lemma}\label{LeFaithfulSubgraph3}
Let $G$ be an infinite locally finite connected graph, let
$\{\mathcal{D}_i\}_{i=1}^{\infty}$ be a sequence of finite sets of
finite connected subgraphs of $G$. Set
$\mathcal{D}=\bigcup_{i=1}^{\infty}\mathcal{D}_i$
and $D=\bigcup\mathcal{D}$. Suppose that\\
(1) $V(G)=\bigcup_{D\in\mathcal{D}}V(D)$ (i.e., $D\unlhd G$);\\
(2) if $j\geq i+2$, then $\bigcup\mathcal{D}_i$ and
$\bigcup\mathcal{D}_j$ are disjoint;\\
(3) $\bigcup\mathcal{D}_1$ is connected, and for every
$D_{j+1}\in\mathcal{D}_{j+1}$, there is a unique
$D_j\in\mathcal{D}_j$ such that $V(D_{j+1})\cap V(D_j)\neq\emptyset$;\\
(4) for every component $H$ of
$G-\bigcup(\bigcup_{i=1}^j\mathcal{D}_i)$, there is a unique
$D_{j+1}\in\mathcal{D}_{j+1}$ with $V(D_{j+1})\cap
V(H)\neq\emptyset$.\\
Then $D\unlhd^FG$.
\end{lemma}

\begin{proof}
We define a sequence of finite subgraphs
$\mathcal{F}=\{F_1,F_2,\ldots\}$ such that
$F_j=\bigcup(\bigcup_{i=1}^{2j-1}\mathcal{D}_i)$. By conditions we
see that $F_j$ is connected, $F_j\leq F_{j+1}$ and
$\bigcup\mathcal{F}=D$.

Let $H$ be a component of $G-F_j$. We will show that $N_D(H)$ is
connected in $F_j-F_{j-1}$ (or $F_1$ if $j=1$). Let $D_{2j}$ be the
unique graph in $\mathcal{D}_{2j}$ with $V(D_{2j})\cap
V(H)\neq\emptyset$, and let $D_{2j-1}$ be the unique graph in
$\mathcal{D}_{2j-1}$ with $V(D_{2j})\cap V(D_{2j-1}\neq\emptyset$.
It follows that $N_D(H)\subseteq V(D_{2j-1})$. Since
$F_{j-1}=\bigcup(\bigcup_{i=1}^{2j-3}\mathcal{D}_i)$,
$V(D_{2j-1})\cap V(F_{i-1})=\emptyset$, i.e., $D_{2j-1}\leq
F_j-F_{j-1}$. Since $D_{2j-1}$ is connected, we see that $N_D(H)$ is
connected in $F_j-F_{j-1}$. By Lemma \ref{LeFaithfulSubgraph2},
$D\unlhd^FG$.
\end{proof}

\begin{theorem}
If $G$ an infinite locally finite connected graph, then $G^2\square
K_2$ has a Hamiltonian circle.
\end{theorem}

\begin{proof}
By Theorem \ref{ThJu} and Lemma \ref{LeNormalFaithful}, $G$ has a
faithful spanning tree $T$. For every $u\in V(G)$, let $N^+(u)$ be
the set of sons of $u$ in $T$, and $d^+(u)=|N^+(u)|$ (so
$d^+(u)=d_T(u)$ if $u$ is the root, and $d^+(u)=d_T(u)-1$
otherwise). We will assign an order on $N^+(u)$, and use $\theta(u)$
to denote the first son of $u$.

For every $u$ with $d^+(u)\neq 0$, we define a finite subtree $T_u$
of $T$ rooted at $u$, and a spanning even semi-cactus $D_u$ of
$T_u^2$ as follows. Set $N^+(u)=\{v_1,v_2,\ldots,v_d\}$ and
$N^+(v_1)=\{w_1,w_2,\ldots,w_{d_1}\}$, where $v_1=\theta(u)$,
$d=d^+(u)$ and $d_1=d^+(v_1)$.\\
(a) If $d=1$, then let $T_u=T[\{u,v_1\}]$ and $D_u=T_u$.\\
(b) If $d\geq 3$ is odd, then let $T_u=T[\{u\}\cup N^+(u)]$, and
$D_u$ be the cycle $uv_1v_2\ldots v_du$.\\
(c) If $d=2$ and $d_1=0$, then let $T_u=T[\{u\}\cup N^+(u)]$,
and $D_u$ be the path $uv_1v_2$.\\
(d) If $d\geq 4$ is even and $d_1=0$, then let $T_u=T[\{u\}\cup
N^+(u)]$, and $D_u$ consist of the edge $uv_1$ and the cycle
$v_1v_2\ldots v_dv_1$.\\
(e) If $d\geq 2$ is even and $d_1\geq 1$ is odd, then let
$T_u=T[\{u\}\cup N^+(u)\cup N^+(v_1)]$ and $D_u$ be the cycle
$uw_1w_2\ldots w_{d_1}v_1v_2\ldots v_du$.\\
(f) If $d=2$ and $d_1\geq 2$ is even, then let $T_u=T[\{u\}\cup
N^+(u)\cup N^+(v_1)]$ and $D_u$ consist of the cycle $uw_1w_2\ldots
w_{d_1}v_1u$ and the edge $v_1v_2$.\\
(g) If $d\geq 4$ is even and $d_1\geq 2$ is even, then let
$T_u=T[\{u\}\cup N^+(u)\cup N^+(v_1)]$ and $D_u$ consist of the two
cycles $uw_1w_2\ldots w_{d_1}v_1u$ and $v_1v_2\ldots v_dv_1$.

Note that all vertices but $\theta(u)$ are d-vertices of $D_u$, and
if $\theta(u)$ is a c-vertex of $D_u$, then $N^+(\theta(u))\subseteq
V(T_u)$. Let $\mathcal{D}_1=\{D_{u_1}\}$, where $u_1$ is the root of
$T$, and for $i=1,2,\ldots$, let
$$\mathcal{D}_{i+1}=\{D_v: d^+(v)\neq 0\mbox{ and }v\mbox{ is a leaf of }T_u\mbox{ with
}D_u\in\mathcal{D}_i\}.$$ Set
$\mathcal{D}=\bigcup_{i=1}^{\infty}\mathcal{D}_i$ and
$D=\bigcup\mathcal{D}$. Note that every vertex is contained in at
most two graphs in $\mathcal{D}$, and if a vertex $v$ is contained
in two graphs in $\mathcal{D}$, then $v$ is a d-vertex of them. We
can see that $D$ is a spanning even semi-cactus of $T^2$.

Note that for every $D_v\in\mathcal{D}_{j+1}$, there is a unique
$D_u\in\mathcal{D}_j$ with $V(D_v)\cap V(D_u)\neq\emptyset$, and for
every component $H$ of $T^2-\bigcup(\bigcup_{i=1}^j\mathcal{D}_i)$,
there is a unique $D_v\in\mathcal{D}_{j+1}$ with $V(D_v)\cap
V(H)\neq\emptyset$. By Lemma \ref{LeFaithfulSubgraph3},
$D\unlhd^FT^2$.

By Lemma \ref{LeFaithfulPower}, $T\unlhd^FT^2$ and $G\unlhd^FG^2$.
Since $T\unlhd^FG$, by Lemma \ref{LeFaithfulSpanning},
$T\unlhd^FG^2$, and by Lemma \ref{LeFaithfulFaithful},
$T^2\unlhd^FG^2$. Again by Lemma \ref{LeFaithfulFaithful},
$D\unlhd^FG^2$. By Lemma \ref{LePrismFaithful}, $D\square
K_2\unlhd^FG^2\square K_2$. By Lemmas \ref{LeSemiHamilton},
\ref{LeSemiCactusHamiltonDegree} and Theorem \ref{ThMain},
$G^2\square K_2$ has a Hamiltonian circle.
\end{proof}

\section{Prisms of line graphs}

In this section we extend Theorem \ref{ThCaKaRoRy2} to infinite
graphs. Let $G$ be an infinite locally finite graph and
$\mathcal{G}=L(G)$ be the line graph of $G$ (i.e.,
$V(\mathcal{G})=E(G)$ and two edges of $G$ are adjacent in
$\mathcal{G}$ if and only if they are adjacent in $G$). Let $R$ be a
ray of $G$, we set $\lambda(R)=L(R)$. So $\lambda(R)$ is an induced
ray of $\mathcal{G}$. Now let $\mathcal{R}=e_1e_2\ldots$ be a ray of
$\mathcal{G}$, we define a ray $\lambda'(\mathcal{R})$ as follows:
First let $\mathcal{R}'=e'_1e'_2\ldots$ be an induced ray of
$\mathcal{G}$ such that $e'_1=e_1$, and for $i=1,2,\ldots$, let
$e_j$ be the last vertex along $\mathcal{R}$ with $e'_ie_j\in
E(\mathcal{G})$, and let $e'_{i+1}=e_{j+1}$. Now let
$R=\lambda'(\mathcal{R})$ be the ray of $G$ such that
$\mathcal{R}'=\lambda(R)$.

By definitions we can see that $\lambda'(\lambda(R))=R$ for every
ray $R$ of $G$, and
$\lambda(\lambda'(\mathcal{R}))\approx_{\mathcal{G}}\mathcal{R}$ for
every ray $\mathcal{R}$ of $\mathcal{G}$.

\begin{lemma}\label{LeMapLineGraph}
Let $G$ be an infinite locally finite connected graph and $\mathcal{G}=L(G)$.\\
(1) If $R_1,R_2$ are two rays of $G$, then $R_1\approx_GR_2$ if and
only if $\lambda(R_1)\approx_\mathcal{G}\lambda(R_2)$.\\
(2) If $\mathcal{R}_1,\mathcal{R}_2$ are two rays of $\mathcal{G}$,
then $\mathcal{R}_1\approx_{\mathcal{G}}\mathcal{R}_2$ if and only
if $\lambda'(\mathcal{R}_1)\approx_G\lambda'(\mathcal{R}_2)$.
\end{lemma}

\begin{proof}
(1) Suppose first that $R_1\approx_GR_2$. Let $\mathcal{S}\subseteq
V(\mathcal{G})$ be arbitrary finite, and let $S$ be the set of
vertices of $G$ that incident to some edges in $\mathcal{S}$. Then
$S$ is finite. It follows that there is a component $H$ of $G-S$
that contains tails of $R_1,R_2$. Let $R'_1,R'_2$ be tails of
$R_1,R_2$ contained in $H$. Clearly $L(H)$ is connected in
$\mathcal{G}-\mathcal{S}$. Let $\mathcal{H}$ be the component of
$\mathcal{G}-\mathcal{S}$ containing $L(H)$. It follows that
$\mathcal{H}$ contains both $\lambda(R'_1),\lambda(R'_2)$, which are
tails of $\lambda(R_1),\lambda(R_2)$, respectively. Therefore
$\lambda(R_1)\approx_\mathcal{G}\lambda(R_2)$.

Suppose now that $R_1\not\approx_GR_2$. Let $S\subseteq V(G)$ be
finite such that $R_1,R_2$ has tails in distinct components of
$G-S$. Let $H_i$ be the component of $G-S$ containing a tail of
$R_i$, $i=1,2$. Let $\mathcal{S}$ be the set of edges that incident
to some vertices in $S$. Then $\mathcal{S}$ is finite. Clearly
$L(H_1)$ and $L(H_2)$ are not connected in
$\mathcal{G}-\mathcal{S}$. It follows that
$\lambda(R_1),\lambda(R_2)$ has tails in distinct components of
$\mathcal{G}-\mathcal{S}$. Therefore
$\lambda(R_1)\not\approx_\mathcal{G}\lambda(R_2)$.

(2) The assertion can be deduced by (1) and the fact that
$\lambda'(\lambda(R))=R$ for every ray $R$ of $G$, and
$\lambda(\lambda'(\mathcal{R}))\approx_{\mathcal{G}}\mathcal{R}$ for
every ray $\mathcal{R}$ of $\mathcal{G}$.
\end{proof}

When concerning more than one graph, we use a subscript to denote
the associated graph of the maps $\lambda,\lambda'$ (i.e.,
$\lambda_G$/$\lambda'_G$ for the maps between the sets of rays of
$G$ and $L(G)$, and $\lambda_F$/$\lambda'_F$ for that of $F$ and
$L(F)$).

\begin{lemma}\label{LeFaithfulLineGraph}
If $F\leq^FG$, then $L(F)\leq^FL(G)$.
\end{lemma}

\begin{proof}
Set $\mathcal{G}=L(G)$ and $\mathcal{F}=L(F)$. Clearly $\mathcal{F}$
is an induced subgraph of $\mathcal{G}$, $\lambda_G(R)=\lambda_F(R)$
for every ray $R$ of $F$, and
$\lambda'_G(\mathcal{R})=\lambda'_F(\mathcal{R})$ for every ray
$\mathcal{R}$ of $\mathcal{F}$.

Let $\mathcal{R}$ be a ray of $\mathcal{G}$ (so
$\lambda'_G(\mathcal{R})$ is a ray of $G$). Since $F\leq^FG$, $F$
has a ray $R$ with $R\approx_G\lambda'_G(\mathcal{R})$. Now
$\lambda_G(R)=\lambda_F(R)$ is a ray of $\mathcal{F}$, and by Lemma
\ref{LeMapLineGraph},
$\lambda_G(R)\approx_{\mathcal{G}}\lambda_G(\lambda'_G(\mathcal{R}))$.
Since
$\mathcal{R}\approx_{\mathcal{G}}\lambda_G(\lambda'_G(\mathcal{R}))$,
we have $\lambda_G(R)\approx_{\mathcal{G}}\mathcal{R}$.

Now let $\mathcal{R}_1,\mathcal{R}_2$ be two rays of $\mathcal{F}$
with $\mathcal{R}_1\approx_{\mathcal{G}}\mathcal{R}_2$. By Lemma
\ref{LeMapLineGraph},
$\lambda'_G(\mathcal{R}_1)\approx_G\lambda'_G(\mathcal{R}_2)$. Since
$F\leq^FG$, and
$\lambda'_G(\mathcal{R}_1)=\lambda'_F(\mathcal{R}_1),\lambda'_G(\mathcal{R}_2)=\lambda'_F(\mathcal{R}_2)$
are two rays of $F$ as well, we have
$\lambda'_F(\mathcal{R}_1)\approx_F\lambda'_F(\mathcal{R}_2)$. Again
by Lemma \ref{LeMapLineGraph},
$\lambda_F(\lambda'_F(\mathcal{R}_1))\approx_{\mathcal{F}}\lambda_F(\lambda'_F(\mathcal{R}_2))$.
Recall that
$\mathcal{R}_1\approx_{\mathcal{F}}\lambda_F(\lambda'_F(\mathcal{R}_1))$
and
$\mathcal{R}_2\approx_{\mathcal{F}}\lambda_F(\lambda'_F(\mathcal{R}_2))$.
We have $\mathcal{R}_1\approx_{\mathcal{F}}\mathcal{R}_2$. This
implies that $\mathcal{F}\leq^F\mathcal{G}$.
\end{proof}

Let $G$ be a graph, let $\mathcal{U}$ be a partition of $V(G)$ such
that each set $U\in\mathcal{U}$ is finite, and let
$\mathcal{G}=G/\mathcal{U}$ be the quotient graph. For every ray
$R=v_1v_2\ldots$ of $G$, we define a ray $\rho(R)=U_1U_2\ldots$ as
follows: Let $U_1$ be the set in $\mathcal{U}$ containing $v_1$. For
$i=1,2,\ldots$, let $v_j$ be the last vertex along $R$ that
contained in $U_i$, and let $U_{i+1}$ be the set in $\mathcal{U}$
containing $v_{j+1}$.

Now for the case that $G[U]$ is finite connected for every
$U\in\mathcal{U}$, we define a ray $\rho'(\mathcal{R})$ of $G$ for
every ray $\mathcal{R}$ of $\mathcal{G}$. We first assign every
$U\in\mathcal{U}$ with a spanning tree $T_U$ of $G[U]$. For every
set $U\in\mathcal{U}$, we choose one vertex in $U$ as its
representative; and for every edge $U_1U_2\in E(\mathcal{G})$, we
choose one edge in $E_G(U_1,U_2)$ as its representative. For a ray
$\mathcal{R}=U_1U_2\ldots$ of $\mathcal{G}$, let $u_1$ be the
representative of $U_1$, $v_iu_{i+1}$ be the representative of
$U_iU_{i+1}$, and $P_i$ be the unique path of $T_{U_i}$ between
$u_i$ and $v_i$. Then $\rho'(\mathcal{R})=u_1P_1v_1u_2P_2v_2\ldots$
is a ray of $G$.

By definitions we can see that $\rho'(\rho(R))\approx_GR$ for every
ray $R$ of $G$, and $\rho(\rho'(\mathcal{R}))=\mathcal{R}$ for every
ray $\mathcal{R}$ of $\mathcal{G}$. When concerning more than one
graph, we use a subscript to denote the associated graph of the maps
$\rho,\rho'$.

\begin{lemma}\label{LeMapQuotientGraph}
Let $G$ be a graph, let $\mathcal{U}$ be a partition of $V(G)$ such
that $G[U]$ is finite connected for each set $U\in\mathcal{U}$, and
let $\mathcal{G}=G/\mathcal{U}$.\\
(1) If $R_1,R_2$ are two rays of $G$, then $R_1\approx_GR_2$ if and
only if $\rho(R_1)\approx_\mathcal{G}\rho(R_2)$.\\
(2) If $\mathcal{R}_1,\mathcal{R}_2$ are two rays of $\mathcal{G}$,
then $\mathcal{R}_1\approx_{\mathcal{G}}\mathcal{R}_2$ if and only
if $\rho'(\mathcal{R}_1)\approx_G\rho'(\mathcal{R}_2)$.
\end{lemma}

\begin{proof}
(1) Suppose first that $R_1\approx_GR_2$. Let $\mathcal{S}\subseteq
V(\mathcal{G})$ be arbitrary finite, and let $S=\bigcup\mathcal{S}$.
Then $S\subseteq V(G)$ is finite. It follows that there is a
component $H$ of $G-S$ that contains tails of $R_1,R_2$. Recall that
each $U\in\mathcal{U}$ induces a finite connected subgraph of $G$.
The subgraph $\mathcal{H}$ of $\mathcal{G}$ induced by
$\{U\in\mathcal{U}: U\subseteq V(H)\}$ is connected in
$\mathcal{G}-\mathcal{S}$. Let $R'_1,R'_2$ be tails of $R_1,R_2$
contained in $H$. Since $\mathcal{H}$ contains all $U$ with $U\cap
V(R'_i)\neq\emptyset$, $\mathcal{H}$ contains a tail of $\rho(R_i)$,
$i=1,2$. It follows that $\rho(R_1)\approx_\mathcal{G}\rho(R_2)$.

Suppose now that $R_1\not\approx_GR_2$. Let $S\subseteq V(G)$ be
finite such that $R_1,R_2$ has tails in distinct components of
$G-S$. Let $H_i$ be the component of $G-S$ containing a tail of
$R_i$, $i=1,2$. Set $\mathcal{S}=\{U\in\mathcal{U}: U\cap
S\neq\emptyset\}$. So $\mathcal{S}$ is a finite subset of
$V(\mathcal{G})$. Clearly any two set $U_1\subseteq V(H_1)$ and
$U_2\subseteq V(H_2)$ are not connected in
$\mathcal{G}-\mathcal{S}$. It follows that $\rho(R_1),\rho(R_2)$ has
tails in distinct components of $\mathcal{G}-\mathcal{S}$. Therefore
$\lambda(R_1)\not\approx_\mathcal{G}\lambda(R_2)$.

(2) The assertion can be deduced by (1) and the fact that
$\rho'(\rho(R))\approx_GR$ for every ray $R$ of $G$, and
$\rho(\rho'(\mathcal{R}))=\mathcal{R}$ for every ray $\mathcal{R}$
of $\mathcal{G}$.
\end{proof}

\begin{lemma}\label{LeFaithfulQuotientGraph}
Let $G$ be a graph, and let $\mathcal{U}$ be a partition of $V(G)$
such that $G[U]$ is finite connected for each set $U\in\mathcal{U}$.
Let $F$ be a spanning subgraph of $G$, and $\mathcal{V}$ be a
subdivision of $\mathcal{U}$ such that for each $V\in\mathcal{V}$,
$F[V]$ is a component of $G[U]$ for some $U\in\mathcal{U}$. Set
$\mathcal{G}=G/\mathcal{U}$ and $\mathcal{F}=F/\mathcal{V}$. If
$F\unlhd^FG$, then $L(\mathcal{F})\leq^FL(\mathcal{G})$.
\end{lemma}

\begin{proof}
Let $\mathcal{R}$ be an arbitrary ray of $L(\mathcal{G})$. Then
$\lambda'_{\mathcal{G}}(\mathcal{R})$ is a ray of $\mathcal{G}$ and
$\rho'_G(\lambda'_{\mathcal{G}}(\mathcal{R}))$ is a ray of $G$.
Since $F\unlhd^FG$, there is a ray $R$ of $F$ with
$R\approx_G\rho'_G(\lambda'_{\mathcal{G}}(\mathcal{R}))$. Clearly
$\lambda_G(\rho_G(R))=\lambda_F(\rho_F(R))$, implying that
$\lambda_G(\rho_G(R))$ is a ray of $L(\mathcal{F})$. By Lemma
\ref{LeMapLineGraph} and \ref{LeMapQuotientGraph},
$\rho_G(R)\approx_{\mathcal{G}}\rho_G(\rho'_G(\lambda'_{\mathcal{G}}(\mathcal{R})))$,
and
$\lambda_{\mathcal{G}}(\rho_G(R))\approx_{L(\mathcal{G})}\lambda_{\mathcal{G}}(\rho_G(\rho'_G(\lambda'_{\mathcal{G}}(\mathcal{R}))))$.
Since
$\mathcal{R}\approx_{L(\mathcal{G})}\lambda_{\mathcal{G}}(\rho_G(\rho'_G(\lambda'_{\mathcal{G}}(\mathcal{R}))))$,
we have
$\lambda_{\mathcal{G}}(\rho_G(R))\approx_{L(\mathcal{G})}\mathcal{R}$.

Now let $\mathcal{R}_1,\mathcal{R}_2$ be two rays of
$L(\mathcal{F})$ with
$\mathcal{R}_1\approx_{L(\mathcal{G})}\mathcal{R}_2$. We will show
that $\mathcal{R}_1\approx_{L(\mathcal{F})}\mathcal{R}_2$. By Lemmas
\ref{LeMapLineGraph} and \ref{LeMapQuotientGraph}, we have
$\rho'_G(\lambda'_{\mathcal{G}}(\mathcal{R}_1))\approx_G\rho'_G(\lambda'_{\mathcal{G}}(\mathcal{R}_2))$.
For $i=1,2$, set
$$S_i=\{v\in V(G): \mbox{there is an edge }U_1U_2\mbox{ of }\mathcal{G}\mbox{ with }U_1U_2\in
V(\mathcal{R}_i)\mbox{ such that }v\in U_1\cup U_2\}.$$ Clearly
$V(\rho'_F(\lambda'_{\mathcal{F}}(\mathcal{R}_i)))\subseteq S_i$,
implying that
$V(\lambda_{\mathcal{G}}(\rho_G(\rho'_F(\lambda'_{\mathcal{F}}(\mathcal{R}_i)))))\subseteq
V(\mathcal{R}_i)\cup N_{L(\mathcal{G})}(\mathcal{R}_i)$. It follows
that
$\lambda_{\mathcal{G}}(\rho_G(\rho'_F(\lambda'_{\mathcal{F}}(\mathcal{R}_i))))\approx_{L(\mathcal{G})}\mathcal{R}_i$.
Recall that
$\lambda_{\mathcal{G}}(\rho_G(\rho'_G(\lambda'_{\mathcal{G}}(\mathcal{R}_i))))\approx_{L(\mathcal{G})}\mathcal{R}_i$.
By Lemmas \ref{LeMapLineGraph} and \ref{LeMapQuotientGraph},
$\rho'_F(\lambda'_{\mathcal{F}}(\mathcal{R}_i))\approx_G\rho'_G(\lambda'_{\mathcal{G}}(\mathcal{R}_i))$.
Thus we have
$\rho'_F(\lambda'_{\mathcal{F}}(\mathcal{R}_1))\approx_G\rho'_F(\lambda'_{\mathcal{F}}(\mathcal{R}_2))$.
Since $F\unlhd^FG$,
$\rho'_F(\lambda'_{\mathcal{F}}(\mathcal{R}_1))\approx_F\rho'_F(\lambda'_{\mathcal{F}}(\mathcal{R}_2))$.
Again by Lemmas \ref{LeMapLineGraph} and \ref{LeMapQuotientGraph},
$\lambda_{\mathcal{F}}(\rho_F(\rho'_F(\lambda'_{\mathcal{F}}(\mathcal{R}_1))))\approx_{L(\mathcal{F})}\lambda_{\mathcal{F}}(\rho_F(\rho'_F(\lambda'_{\mathcal{F}}(\mathcal{R}_2))))$.
Recall that
$\mathcal{R}_i\approx_{L(\mathcal{F})}\lambda_{\mathcal{F}}(\rho_F(\rho'_F(\lambda'_{\mathcal{F}}(\mathcal{R}_i))))$,
$i=1,2$, we have
$\mathcal{R}_1\approx_{L(\mathcal{F})}\mathcal{R}_2$. Therefore
$L(\mathcal{F})\leq^FL(\mathcal{G})$.
\end{proof}

\begin{lemma}\label{LePendantRemove}
Let $G$ be an infinite locally finite graph, and $F$ be the graph
obtained from $G$ by removing some pendant vertices. Then
$F\leq^FG$.
\end{lemma}

\begin{proof}
The assertion is deduced by the fact that for every ray $R$ of $G$,
all but at most the first vertex and edge of $R$ are contained in
$F$.
\end{proof}

An \emph{essential} cut-edge of graph $G$ is one whose removal
products at least two non-trivial components. Note that $L(G)$ is
2-connected if and only if $G$ has no essential cut-edge.

\begin{lemma}\label{LeBlockSemiCactus}
Let $G$ be an infinite locally finite connected graph such that\\
(1) every block of $G$ is either a cycle or a $K_2$ (i.e., $G$
contains no $\varTheta$-graph);\\
(2) if a vertex $v$ is not contained in a cycle, then all but at
most two neighbors of $v$ are pendant.\\
Then $L(G)$ has a faithful spanning even semi-cactus.
\end{lemma}

\begin{proof}
Suppose first that $G$ has no $K_2$-block. Choose an arbitrary edge
$uv\in E(G)$. We cleave $u$ to $u_1,u_2$ such that $u_1$ adjacent to
$v$ and $u_2$ is adjacent to all other neighbors of $u$. Let $F$ be
obtained from the resulting graph by removing the edge $u_1u_2$.
Then $F$ satisfies the conditions (1)(2) and has a $K_2$-block.
Moreover, by Lemma \ref{LeFaithfulQuotientGraph} (take $\mathcal{U}$
as one set $\{u_1,u_2\}$ together will all singletons consist with a
vertex in $V(G)\backslash\{u\}$), we see that $L(F)\unlhd^FL(G)$.
Thus, it sufficiency to consider the case that $G$ has some
$K_2$-blocks.

Let $\mathcal{B}$ be the set of blocks of $G$ and $U$ be the set of
c-vertices of $G$. We define a graph $\mathcal{T}$ on
$\mathcal{B}\cup U$ such that for any $B\in\mathcal{B}$ and $u\in
U$, $Bu\in E(\mathcal{T})$ if and only if $u\in V(B)$. Clearly
$\mathcal{T}$ is a tree, and we take a $K_2$-block as the root of
$\mathcal{T}$. We assign an orientation to each cycle of $G$, and
for each $u\in U$, we assign an order to $N^+(u)$ (the set of sons
of $u$ in $\mathcal{T}$) in such a way that the cycle-blocks always
appear before the $K_2$-blocks, and the $K_2$-blocks which are
essential cut-edges always appear before the pendant edges.

Let $B\in\mathcal{B}$ which is not the root. We define some
associated vertices $u_B,v_B,w_B$, edges $\varepsilon_B,\epsilon_B$,
and blocks $\varPhi_B,\varPsi_B$ as follows. Let $v_B$ be the father
of $B$ in $\mathcal{T}$. If $B$ is the first son of $v_B$, then let
$\varPhi_B$ be the father of $v_B$; otherwise let $\varPhi_B$ be the
son of $v_B$ that exactly elder than $B$. If $\varPhi_B$ is a $K_2$,
then let $u_B$ be the vertex of $B$ other than $v_B$; if $\varPhi_B$
is a cycle, then let $u_B$ be the predecessor of $v_B$ on
$\varPhi_B$. Let $\varepsilon_B=u_Bv_B$. If $B$ is a $K_2$, then let
$w_B$ be the vertex of $B$ other than $v_B$; if $B$ is a cycle, then
let $w_B$ be the successor of $v_B$ on $B$. Let $\epsilon_B=v_Bw_B$.
If $B$ is a cycle and $w_B$ is contained in a cycle other than $B$,
then let $\varPsi_B$ be the first son of $w_B$ (which is a
cycle-block by our order on $N^+(w_B)$); otherwise $\varPsi_B$ is
not defined.

Note that the root of $\mathcal{T}$ is $\varPhi_B$ of at most two
blocks $B$, and its edge is $\epsilon_B$ of at most two blocks $B$.
Apart from this, every block is $\varPhi_B$ of at most one block
$B$, and is $\varPsi_B$ of at most one block $B$, and every edge is
$\varepsilon_B$ of at most one block $B$, and is $\epsilon_B$ of at
most one block $B$.

Now we define a finite subgraph $G_B$ of $G$ and a spanning even
semi-cactus $D_B$ of $L(G_B)$ as follows.\\
(a) If both $B$ and $\varPhi_B$ are $K_2$'s, then $G_B$ consists
of $B$ and all $K_2$-blocks attached to $v_B$.\\
(b) If $B$ is a $K_2$ and $\varPhi_B$ is a cycle, then $G_B$
consists of $B$ and $\varepsilon_B$.\\
(c) If $B$ is a cycle, $\varPhi_B$ is a $K_2$ and $\varPsi_B$ does
not exist, then $G_B$ consists of $B$ and all $K_2$-blocks attached
to some vertex in $V(B)$.\\
(d) If both $B$ and $\varPhi_B$ are cycles and $\varPsi_B$ does not
exist, then $G_B$ consists of $B$, $\varepsilon_B$ and all
$K_2$-blocks attached to some vertex in $V(B)\backslash\{v_B\}$.\\
(e) If $B$ is a cycle, $\varPhi_B$ is a $K_2$ and $\varPsi_B$
exists, then $G_B$ consists of $B$, $\varPsi_B$ and all $K_2$-blocks
attached to some vertex in $V(B)\cup V(\varPsi_B)$.\\
(f) If both $B$ and $\varPhi_B$ are cycles and $\varPsi_B$ exists,
then $G_B$ consists of $B$, $\varepsilon_B$, $\varPsi_B$ and all
$K_2$-blocks attached to some vertex in $V(B)\cup
V(\varPsi_B)\backslash\{v_B\}$.\\
Note that $\varepsilon_B$ is always contained in $G_B$, and if an
edge $e\neq\varepsilon_B$ is contained in $G_B$, then the block
containing $e$ is also contained in $G_B$.

If $B$ is a $K_2$, then $G_B$ is a star. For this case let $D_B$ be
a Hamiltonian path of $L(G_B)$ between $\varepsilon_B$ and
$\epsilon_B$. For the case $B$ is a cycle, noting that
$B\cup\varPsi_B$ (or $B$ if $\varPsi_B$ does not exist) is a
dominating closed trail of $G_B$, both $L(G_B)$ and
$L(G_B)-\varepsilon_B$ are Hamiltonian. If $|E(G_B)|$ is even, then
let $D_B$ be a Hamiltonian cycle of $L(G_B)$; if $|E(G_B)|$ is odd,
then let $D_B$ consists of $\varepsilon_B\epsilon_B$ and a
Hamiltonian cycle of $L(G_B)-\varepsilon_B$.

Let $\mathcal{B}_1=\{B: \varepsilon_B\mbox{ is the root of
}\mathcal{T}\}$; and for $i=1,2,\ldots$, let
$$\mathcal{B}_{i+1}=\{B: \varepsilon_B\in E(G_{B'})\mbox{ for some
}B'\in\mathcal{B}_i\mbox{ and }B\mbox{ is not contained in
}G_{B''}\mbox{ for any }B''\in\mathcal{B}_i\}.$$ Set
$\mathcal{G}_i=\{G_B: B\in\mathcal{B}_i\}$,
$\mathcal{G}=\bigcup_{i=1}^{\infty}\mathcal{G}_i$,
$\mathcal{D}_i=\{D_B: B\in\mathcal{B}_i\}$,
$\mathcal{D}=\bigcup_{i=1}^{\infty}\mathcal{D}_i$ and
$D=\bigcup\mathcal{D}$. By definition we see that every edge $e$ of
$G$ is contained in at most two graphs in $\mathcal{G}$, and if $e$
is contained in two graphs in $\mathcal{G}$, say
$G_{B_i}\in\mathcal{G}_i$ and $G_{B_{i+1}}\in\mathcal{G}_{i+1}$,
then $e=\varepsilon_{B_{i+1}}$. Clearly $D\unlhd L(G)$.

We can see that for every block $B_{i+1}\in\mathcal{B}_{i+1}$, there
is a unique graph $B_i\in\mathcal{B}_i$ such that
$E(G_{B_{i+1}})\cap E(G_{B_i})\neq\emptyset$. In fact $G_{B_i}$ is
the graph that containing $\varPhi_B$, and $E(G_{B_{i+1}})\cap
E(G_{B_i})=\{\varepsilon_B\}$.

Recall that the root of $\mathcal{T}$ is a c-vertex of graphs in
$\mathcal{D}_1$. We claim that if an edge is contained in two graphs
in $\mathcal{D}$, then it is a d-vertex of both graphs. Suppose $e$
is a c-vertex of $D_B$. Then either both $B$ and $\varPhi_B$ are
$K_2$'s, and $e$ is in a $K_2$-block attached to $v_B$ other than
$B$ and $\varPhi_B$; or $B$ is a cycle and $e=\epsilon_B$. If both
$B$ and $\varPhi_B$ are $K_2$'s, and $e$ is in a $K_2$-block
attached to $v_B$ other than $B$ and $\varPhi_B$, then $v_B$ is not
contained in a cycle; otherwise $G_B\notin\mathcal{G}$. If $e$ is a
essential cut-edge of $G$, then so is $\epsilon_B$ by our order on
$N^+(v_B)$, and $v_B$ has three neighbors which are c-vertices of
$G$, a contradiction. Thus we have that $e$ is a pendant edge of
$G$, implying that $e\neq\varepsilon_{B'}$ for any block
$B'\in\bigcup_{i=1}^{\infty}\mathcal{B}_i$. Suppose now that $B$ is
a cycle and $e=\epsilon_B$. In this case $\varPsi_B$ is the only
possible block $B'$ with $e=\varepsilon_{B'}$. But now $B'$ is
contained in $G_B$ and $G_{B'}\notin\mathcal{G}$. In both case $e$
is contained in exactly one graphs in $\mathcal{D}$, as we claimed.

We can see that $D$ is a spanning even semi-cactus of $L(G)$.

Note that $\mathcal{D}_1$ consists of exactly one or two graphs, and
if $|\mathcal{D}_1|=2$, then they are common to the root of
$\mathcal{T}$. For every $D_{B_{i+1}}\in\mathcal{D}_{i+1}$, there is
a unique $D_{B_i}\in\mathcal{D}_i$ with $V(D_{B_{i+1}})\cap
V(D_{B_1})\neq\emptyset$, and for every component $H$ of
$L(G)-\bigcup(\bigcup_{i=1}^j\mathcal{D}_i)$, there is a unique
$D_{B_{i+1}}\in\mathcal{D}_{i+1}$ with $V(D_{B_{i+1}})\cap
V(H)\neq\emptyset$. By Lemma \ref{LeFaithfulSubgraph3},
$D\unlhd^FL(G)$.
\end{proof}

\begin{theorem}
Let $G$ be an infinite locally finite graph. If $L(G)$ is
2-connected, then $L(G)\square K_2$ has a Hamiltonian circle.
\end{theorem}

\begin{proof}
By condition $G$ has no essential cut-edge. Let $V_1$ be the set of
vertices of degree 1 in $G$, and let $G'=G-V_1$. Thus $G'$ is
2-edge-connected. By Lemma \ref{LePendantRemove}, $G'\leq^FG$, and
by Lemma \ref{LeFaithfulLineGraph}, $L(G')\leq^FL(G)$.

For a vertex $u\in V(G')$ with $d_{G'}(u)\geq 4$, we cleave $u$ such
that each of $u_1,u_2$ is adjacent to at least two neighbors of $u$.
Moreover, we can always do this in such way the the added edge
$u_1u_2$ is not a cut-edge. We repeatedly do this on every vertex of
degree at least 4. So the resulting graph $G''$ is a
2-edge-connected subcubic graph (and thus, is 2-connected). Every
vertex $u\in V(G')$ is cleaved to a tree $T_u$ in $G''$ ($T_u$ is
trivial if $d_{G'}(u)\leq 3$). Set $\mathcal{U}=\{V(T_u): u\in
V(G')\}$. Then $G'=G''/\mathcal{U}$.

By Lemma \ref{LeFaithfulCactus}, $G''$ has a faithful spanning
cactus $F''$. Let $\mathcal{V}$ be a subdivision of $\mathcal{U}$
such that for each $V\in\mathcal{V}$, $F''[V]$ is a component of
$G''[U]$ for some $U\in\mathcal{U}$. Let $F'=F''/\mathcal{V}$. By
Lemma \ref{LeFaithfulQuotientGraph}, $L(F')\leq^FL(G')$.

Note that every $V\in\mathcal{V}$ induces a tree of $F''$. For each
two set $V_1,V_2\in\mathcal{V}$, $V_1,V_2$ are separable by an
edge-cut of size $k$ in $F''$ if and only if they are separable by
an edge-cut of size $k$ in $F'$. Moreover, if $V\cap
V(B)\neq\emptyset$ for some cycle-block of $F''$, then $V$ is
contained in some cycle of $F'$. It follows that every block of $F'$
is either a $K_2$ or a cycle, and if a vertex $V$ of $F'$ is not
contained in a cycle-block, then $d_{F'}(V)\leq 2$.

Set $\mathcal{U}_{\mathcal{V}}=\{\{V\in\mathcal{V}: V\subseteq U\}:
U\in\mathcal{U}\}$. It follows that every set of
$\mathcal{U}_{\mathcal{V}}$ is an independent set of $F'$ and
$F'/\mathcal{U}_{\mathcal{V}}\unlhd G'$. Let $e$ be an arbitrary
vertex of $G$ that not contained in $F'/\mathcal{U}_{\mathcal{V}}$.
Then $e$ is incident to at least one vertex $u\in V(G')$. Let $V$ be
a vertex of $F'$ with $V\subseteq V(T_u)$. We attach the edge $e$ to
$V$ (i.e., $e$ will be a pendant edge). Let $F$ be the graph
obtained from $F'$ by attaching all edges in $E(G)\backslash
E(F'/\mathcal{U}_{\mathcal{V}})$. Thus $L(F)\unlhd L(G)$, and all
attached edges are pendant in $F$. By Lemma \ref{LePendantRemove}
and Lemma \ref{LeFaithfulLineGraph}, $L(F')\leq^FL(F)$.

Recall that every block of $F$ is either a $K_2$ or a cycle. For any
vertex $V$ of $F'$ that is not contained in a cycle, $V$ has at most
two neighbors in $F'$, and all vertices in $N_F(V)\backslash
N_{F'}(V)$ is pendant in $F$. It follows that all but at most two
neighbors of $V$ are c-vertices of $F$. By Lemma
\ref{LeBlockSemiCactus}, there is an even semi-cactus $D$ with
$D\unlhd^FL(F)$.

By Lemma \ref{LeFaithfulFaithful}, and the fact $L(G')\leq^FL(G)$,
$L(F')\leq^FL(F)$, $L(F')\leq^FL(G')$, we have $L(F)\leq^FL(G)$.
Since $L(F)\unlhd L(G)$, $L(F)\unlhd^FL(G)$. By Lemma
\ref{LeFaithfulFaithful}, and the fact $D\unlhd^FL(F)$, we have
$D\unlhd^FL(G)$. By Lemma \ref{LeSemiCactusHamiltonDegree} and
Theorem \ref{ThMain}, $G$ has a Hamiltonian circle.
\end{proof}

\end{document}